\documentclass[12pt]{article}

\usepackage{amsmath,amsthm,amsfonts,amssymb,mathrsfs,dsfont,bbm}
\usepackage[usenames]{color}

\usepackage{amsthm,amssymb,amsmath,amsbsy}
\usepackage{color}

\definecolor{labelkey}{rgb}{0.0, 0.8, 0.3}

\usepackage{hyperref}
\hypersetup{
  colorlinks = true,
  urlcolor = blue,
  linkcolor = blue,
  citecolor = blue}

\newcommand{\cD}{\mathcal{D}}

\newcommand{\cG}{\mathcal{G}}

\newcommand{\cM}{\mathcal{M}}

\newcommand{\cS}{\mathcal{S}}
\newcommand{\cT}{\mathcal{T}}

\newcommand{\cZ}{\mathcal{Z}}

\newcommand{\bt}{\overline{\theta}}
\newcommand{\bp}{\overline{\varphi}}
\newcommand{\tlt}{\tilde{\theta}}
\newcommand{\tlp}{\tilde{\varphi}}
\newcommand{\ind}{\mathbbm{1}}

\newcommand\numberthis{\addtocounter{equation}{1}\tag{\theequation}}

\numberwithin{equation}{section}

\definecolor{ggreen}{rgb}{0.0, 0.5, 0.3}
\definecolor{rred}{rgb}{0.65, 0.2, 0.2}
\definecolor{bblue}{rgb}{0.0, 0.0, 1}

\newtheorem{thm}{Theorem}
\newtheorem*{thm*}{Theorem}
\newtheorem{rem}[thm]{Remark}
\newtheorem{defi}[thm]{Definition}
\newtheorem{lem}[thm]{Lemma}
\newtheorem{cor}[thm]{Corollary}
\newtheorem{prop}[thm]{Proposition}

\newcommand{\R}{\mathbb R}
\newcommand{\E}{\mathbb E}
\newcommand{\Z}{\mathbb Z}
\newcommand{\C}{\mathbb C}

\newcommand{\eps}{\varepsilon}

\renewcommand{\phi}{\varphi}
\newcommand{\ud}{\mathrm{d}}
\renewcommand{\t}{\theta}
\newcommand{\symm}{\mathrm{symm}}

\DeclareMathOperator{\supp}{supp}
\DeclareMathOperator{\psupp}{psupp}

\definecolor{pink}{cmyk}{0, 1, 0, 0}

\def\ct{\check{\theta}}
\def\ch{\check{h}}
\def\hht{\hat{\theta}}
\def\hh{\hat{h}}
\def\hf{\hat{f}}

\def\fS{\mathfrak{S}}

\def\fb{\Gamma}

\def\a{\alpha}
\def\b{\beta}

\def\s{\sigma}
\def\k{\kappa}

\def\arg{\text{Arg}}

\def\ol{\overline}

\def\c{\complement}
\def\lg{\langle}
\def\rg{\rangle}

\def\L{\Lambda}

\def\var{\mathrm{Var}}

\def\bG{\mathcal{G}}

\def\cC{\mathcal{V}}
\def\Del{\Delta}

\def\ol{\overline}
\def\P{\mathbb{P}}
\def\o{\omega}

\def\vp{\varphi}
\def\vr{\varrho}
\def\vs{\zeta}

\def\fF{\mathfrak{F}}

\def\fm{\mathfrak{m}}
\def\fa{\mathfrak{a}}

\makeatletter
\renewcommand*{\@cite@ofmt}{\hbox}
\makeatother

\begin{document}

\title{\Large  Sparse Multi-Reference Alignment : Phase Retrieval, \\ Uniform  Uncertainty Principles and the Beltway Problem}
\author{
			{Subhroshekhar Ghosh} \thanks{Dept. of Mathematics, National University of Singapore, \texttt{subhrowork@gmail.com}}
		\and
			{Philippe Rigollet} \thanks{Dept. of Mathematics, MIT, \texttt{rigollet@math.mit.edu}}
	}
\date{}
\maketitle

\begin{abstract}
{\small 
Motivated by cutting-edge applications like cryo-electron microscopy (cryo-EM), the Multi-Reference Alignment (MRA)  model  entails the learning of an unknown signal from repeated measurements of its images under the latent action of a group of isometries and additive noise of magnitude $\sigma$. Despite significant interest, a clear picture for understanding rates of estimation in this model has emerged only recently, particularly in the high-noise regime $\sigma \gg 1$ that is highly relevant in applications. Recent investigations have revealed a remarkable asymptotic sample complexity of order $\sigma^6$  for certain signals whose Fourier transforms have full support, in stark contrast to the traditional $\sigma^2$ that arise in regular models. Often prohibitively large in practice, these results have prompted the investigation of variations around the MRA model where better sample complexity may be achieved. In this paper, we show that \emph{sparse} signals exhibit an intermediate $\sigma^4$ sample complexity even in the classical MRA model. Further, we characterise the dependence of the estimation rate on the support size $s$ as $O_p(1)$ and $O_p(s^{3.5})$ in the dilute and moderate regimes of sparsity respectively. Our techniques have implications for the problem of \textit{crystallographic phase retrieval}, indicating a certain local uniqueness for the recovery of sparse signals from their power spectrum. 
Our results explore and exploit 
connections of the MRA estimation problem with two classical topics in applied mathematics: the \textit{beltway problem} from combinatorial optimization, and \textit{uniform uncertainty principles} from harmonic analysis.} Our techniques include a certain enhanced form of the probabilistic method, which might be of general interest in its own right.
\end{abstract}

\tableofcontents

\section{Introduction}
\subsection{The MRA problem}
The Multi Reference Alignment (MRA) problem is a simple model that captures fundamental characteristics of various statistical models with latent group actions. It arises in various questions across science and engineering such as structural biology  ~\cite{Sad89,Dia92,SchValNun05,ParMidMad11,TheSte12,ParChi14}, image recognition~\cite{Bro92,DryMar98,ForZerBer02,RobFarMil07}, robotics~\cite{RosCarBan19} and signal processing~\cite{Zvd03,PZAF05}. This problem also serves as a simplification for more complex ones that feature repeated observations of a signal subject to latent group actions and additive measurement noise. Such problems include, for example, the three-dimensional reconstruction of molecules using cryo-electron microscopy (cryo-EM)~\cite{BarMerBan15,SirCheSun16,PerWeeBan19}. Such models have gained salience in recent years with the remarkable growth in the scope and capabilities of data-intensive procedures in science and technology. 

The MRA problem ~\cite{BanChaSin14,Rit89,PerWeeBan19} consists in a signal $\t: \Z_L \mapsto \R$ (equivalently, a vector $\t \in \R^L$), and $n$ independent noisy observations  $y_1,\cdots,y_n$ that satisfy
\begin{equation} \label{eq:MRA-model}
y_i=R_{i} \t + \sigma \xi_i,
\end{equation}
where the $R_{i}$-s are isometries of $\R^L$, and the random variables $\xi_i$ are i.i.d. $L$-dimensional standard Gaussians $N(0,I_L)$, and $\sigma>0$ is the scale of the noise. The $R_{i}$-s are taken random, sampled from the group of cyclic shifts $\cG$ on $\R^L$, and are  independent as random variables from the noise $\{\xi_i\}_i$. 

The group of cyclic co-ordinate shifts is given by $(R_{\ell} \t)_k=\t_{k+\ell \; (\text{mod} L)}$, where $(v)_k$ denotes the $k$-th co-ordinate for a vector $v \in \R^L$. 
The canonical distribution for the isometries $R_{i}$ is uniform on the group  $\cG$, although other distributions have been considered \cite{AbbBenLee19}.

Of course, due to the latent group actions, it is not possible to recover $\theta$ exactly. Instead, our goal is to obtain an estimator $\tilde \theta$ whose distance to the orbit of $\theta$ under the action of the group $\cG$, as defined by
\begin{equation} \label{eq:varrho}
\varrho(\tilde \theta, \theta):= \min_{G \in \cG} \frac{1}{\sqrt{L}} \|\tilde \theta - G \theta\|_2
\end{equation}
is typically small with growing number of samples $n$.

On a related note, we also define the distance $\rho$ below, which will enable us to invoke results from the literature on the MRA model.
\begin{equation} \label{eq:def_rho}
\rho(\t , \phi):= \min_{G \in \cG}  \|\t - G \phi\|_2
\end{equation}
Observe that $\varrho(\t,\phi)=\frac{1}{\sqrt{L}}\rho(\t,\phi)$; so results in the two metrics are simple scalings of each other by a factor of $\sqrt{L}$.

In this work, we focus on the statistical performance achievable in the MRA model. Of key interest is the dependence on the typical behaviour of $\varrho$ on the quantities $n$ and $\sigma$  for the asymptotics $n, \sigma \to \infty$ which are well justified by applications such as molecular spectroscopy~\cite{PerWeeBan19,Sig98}. In this regime, statistical rates of convergence are of the form
$$
\E\varrho(\tilde \theta, \theta)\le C(L,s) \frac{\sigma^\alpha}{\sqrt n}
$$
where $\alpha$ is an exponent that critically controls the performance of $\tilde \theta$ in the regime of interest~\cite{BanRigWee17,PerWeeBan19}. We also keep track of other important quantities such as the dimension $L$ or the sparsity $s$ of the signal $\theta$ but only insofar as they appear in leading terms. Dual to the above \emph{rate}, one may consider the \emph{sample complexity} of $\tilde \theta$, which is the number $n$ of samples required to achieve accuracy $\eps$. Equating the right-hand side of the ablve display with $\eps$ and solving for $n$ yields a sample complexity of $\sigma^{2\alpha}/\eps^2$. In this work, we always achieve parametric rates where the dependence on $n$ is $n^{-1/2}$ and hence, the sample complexity thus scales as $\eps^{-2}$ in $\eps$. As a result, we refer to $\sigma^{2\alpha}$ as the sample complexity of $\tilde \theta$.


While the MRA problem has been mostly attacked using the synchronization approach  \cite{BanChaSin14}, it is only recently that it was recognized as a Gaussian mixture model~\cite{BanRigWee17} which has enabled the use of various methods such as the method of moments~\cite{PerWeeBan19,BouBenLed18} and expectation-maximization~\cite{BenBouMa17} to recover the signal of interest $f$.  For a detailed discussion on the likelihood landscape of such models, we refer the reader to  ~\cite{FanSunWan20,Bru19,fan2021maximum,katsevich2020likelihood}.

Using the Gaussian structure of the noise, it is straightforward to write down an expression for the likelihood function for given observations $\{y_1,\cdots,y_n\}$, where the likelihood function is parametrized by the signal parameter $\t$. If the density corresponding to $\t$ for an observation $y$ is given by $p_\t$, then we can write
\begin{equation} \label{eq:density}
p_\t(y)=\frac{1}{|\cG|} \sum_{R \in \cG} \frac{1}{(\sqrt{2\pi} \sigma)^L} \exp{\left( - \frac{\|y-R\t\|_2^2}{2\sigma^2}\right)}
\end{equation}
and the log likelihood corresponding to the data $\{y_1,\ldots,y_n\}$ as
\begin{equation} \label{eq:log-likelihood}
\mathcal{L}(\t)=\sum_{i=1}^n \log p_{\t}(y_i).
\end{equation}

The perspective of Gaussian mixture models has  enabled the discovery of a singular statistical phenomenon due to the presence of the latent isometries \cite{BanRigWee17,PerWeeBan19}. To recall this result, we introduce some notation.

We consider vectors in $\R^L$ as functions mapping $\Z_L$ to $\R$, and  consider $\Z_L$ in the standard parametrization \ref{eq:std-par}.  Let $\hat \theta \in \R^L$ denote the (discrete) Fourier transform of $\theta$, also considered as a function $\hht:\Z_L \mapsto \R$, where $\Z_L$ is viewed in the standard parametrization. Since the signal $\t$ is real, $\hht$ is symmetric about the origin. We define the positive support of $\hht$ by
$$
\psupp(\hat \theta) = \{j \mid j \in \{1, \cdots,  \lfloor (L-1)/2 \rfloor \}, \hat \theta_j \neq 0\}\,.
$$
To prohibit $\hat \theta$ to scale with the sample size $n$, it is reasonable to assume that there exists two positive constants $c$ and $c_0$, such that $c^{-1} \le \|\theta\|\le c$ and $|\hat \theta_j|\ge c_0$ for all $j \in \psupp(\hat \theta)$. The group action under consideration is the group of shifts $\Z_L$. 

We now discuss the minimax lower bound proved in~\cite{BanRigWee17}, which is also shown to give the optimal rate. To be precise, the results in ~\cite{BanRigWee17} are stated in the setting of the closely-related phase-shift model (essentially, a continuous version of the MRA), but the broad implications of the result are understood to also capture the behaviour of the MRA model. \cite{BanRigWee17} gives us the following minimax lower bound on the estimation error.
\begin{thm}\cite[Theorem~1]{BanRigWee17}
\label{thm:lower}
Let $2 \leq s \leq L/2$.
Let $\cT_s$ be the set of vectors $\theta \in \cT$ satisfying %
$\psupp(\hat \theta) \subset [s]$.
For any $\sigma \geq \max_{\theta \in \cT_s} \|\theta\|$, the phase shift model satisfies
\begin{equation}\label{eqn:minimax_lower}
\inf_{T_n} \sup_{\theta \in \cT_s} \E[\varrho(T_n, \theta)] \gtrsim  \frac{\sigma^{2s-1}}{\sqrt n} \wedge 1\,,
\end{equation}
where the infimum is taken over all estimators $T_n$ of $\theta$.
\end{thm}
A careful inspection of the proof of this lower bound indicates that it is in fact driven by specific cancellations of the Fourier coefficients of $\theta$. Indeed, if $\psupp(\hat \theta)\subset[(L-1)/2]$, there exists specific sparsity patterns for the Fourier transform of $\theta$ that make it hard to estimate: in this case, Theorem~\ref{thm:lower} indicates a worst-case lower bound with a terrible sample complexity: $\sigma^{2L-2}$. This result is mitigated  in~\cite{PerWeeBan19} where it is proved that if $\psupp(\hht)=[L/2]$, that is if $\hat \theta$ has \emph{full support}---recall that we assumed $|\hat \theta_j|\ge c_0$ for all $j \in \supp(\hat \theta)$---then a sample complexity of $\sigma^6$ may be achieved. The unusual exponent $6=2\cdot 3$ comes from the fact that in this case, the orbit of $\theta$ may be identified from the first three moments of $Y$.

While $\sigma^6$ is a significant improvement over  $\sigma^{2L-2}$, this  scaling is still inauspicious in applications where $\sigma$ is large (notice that the dependence of sample complexity on $\s$ scales like the square of that of the estimation rate as in \eqref{eqn:minimax_lower}). This situation has prompted the investigation of settings where the orbit of $\theta$ could be recovered robustly only from its first two moments, thus leading to the a sample complexity $\sigma^4$. This is the case for example if $\hat \theta$ has full support and the distribution of the isometries on the group $\cG$ is \textit{not uniform} but follows some specific distribution instead~\cite{AbbBenLee19}.

In this paper, we unveil generic conditions on the signal $\theta$ under which a sample complexity of $\sigma^4$ can be achieved in the original MRA model, where the distribution of isometries from the group $\cG$ \textit{is uniform}. Interestingly, these results are built on connections with other well-studied questions in applied mathematics, in particular the \emph{beltway problem} from combinatorial optimization and \emph{uniform uncertainty principles} from harmonic analysis.

In methodological terms, obtaining a $\s^2/\sqrt{n}$ rate will be found to be related to our ability to recover a signal from the second moment tensor, and  in turn, from the modulus of the Fourier coefficients of its observations in the MRA model. This  will eventually be made possible by the sparsity of the signal. In a related vein, we note that the well-known problem of phase retrieval, albeit in a  different context, examines signal recovery from the modulus of its random linear measurements. However, it may be noted that our observational setting in the MRA model with the latent group action has a very different and more complicated structural setup than the existing literature on phase retrieval, which largely focuses on a specific setting akin to compressed sensing with limited information. Nonetheless, there are natural connections to phase retrieval, especially to the so-called \textit{crystallographic phase retrieval} problem \cite{bendory2020toward}; this is discussed in detail in Section \ref{sec:Phase_Retrieval}.

\subsection{Estimation rates for generic sparse signals}

In the present work, we focus attention on the class of sparse signals in the context of the MRA problem, and explore rates if estimation for such signals. Our investigations naturally demarcate the set of sparse signals into two regimes, marked by results of differing nature. 

On one hand, we have the \textit{dilute} regime of sparsity (roughly, of the order $L^{1/3}$), where a randomly chosen subset of $\Z_L$ of that size does not have have any multiplicities in its mutual differences. This condition is referred to as the collision-free property of the subset. In the dilute regime of sparsity, we establish $O(\sigma^4)$ sample complexity. This is complemented by the \textit{moderate} regime of sparsity, which extends all the way up to order $L/\log^5L$ where we show that the improved sample complexity may also be achieved.

We also unveil the dependence structure of the  estimation rate asymptotics on the sparsity $s$ of the signal. In the dilute regime, there is an $O_p(1)$ dependency; whereas in the moderate regime of sparsity, the dependency is $O_p(s^{3.5})$. 
Observe that  we are considering asymptotic rates of estimation which are by nature local to the true signal; this is different from non-asymptotic rates which usually involve  additional logarithmic dependence on $L$.

The relative difficulty in obtaining $O(\sigma^4)$ sample complexity with increasing size of the signal support is reflected in the dependence structure of the asymptotics on the sparsity, as well as in the additional assumptions required in the moderate regimes. 
Such behaviour is perhaps well anticipated in view of the fact that, in the regime of full signal support, sample complexity of order better than $\sigma^6$ generically not possible, a result which we also establish in this work.

In the dilute regime, our methodological ingredients include exploiting collision-freeness, whereas in the moderate regime they include  repeated, nested applications of the probabilistic method in order to find frequency sets conducive to our analysis in the Fourier space, aided by the tool of uniform uncertainty principles.

In order to discuss our results in detail, we first introduce a few notations and concepts.
\subsubsection{Some notations and concepts} \label{sec:notn}
In this work, we will set $\cG$ to be the group of rotations  by the elements of $\Z_L$, that is, for each $g \in \cG$ and $v: \Z_L \mapsto \C$, we define the action $[g \cdot v]$ as $[g \cdot v](i)=v(i+g) \quad \forall i \in \Z_L$. 

We note in passing that the results of this paper would also hold under the action of a richer group of isometries $\cG$ where the rotations of $\Z_L$ are augmented with a reflection or ``flip'', that is, the operation $\alpha$ acting on $\Z_L$ that sends  $x \mapsto \check{x}= -x$; in other words, the group $\cG=\Z_L \rtimes \Z_2$. In fact, there has been recent interest focussed on \textsl{dihedral} multi reference alignment \cite{bendory2022dihedral}. However, for purposes of presentation, we will adhere to the isometry group $\cG$ given by the rotations of $\Z_L$.

We view the signal $\t$ as a function  on the discretised circle $\Z_L$, where the elements of the latter are enumerated as 
\begin{equation} \label{eq:std-par}
\Z_L=\{\lfloor - (L-1)/2 \rfloor,  \lfloor - (L-1)/2 \rfloor+1, \cdots, \lfloor (L-1)/2 \rfloor-1, \lfloor (L-1)/2 \rfloor \}
\end{equation}
We call this parametrization the \textit{standard parametrization} of $\Z_L$. The \textit{positive} part of $\Z_L$ may then be defined as 
\begin{equation} \label{eq:std-par-pos}
\Z_L^+=\{0,1, \cdots, \lfloor (L-1)/2 \rfloor-1, \lfloor (L-1)/2 \rfloor \}.
\end{equation}

We include here a discussion on restricted MLEs, which will constitute the main estimators in describing our statistical results. For a deterministic set of signals $\cT$ (where the true signal is known to belong), it is natural to maximise the log likelihood \eqref{eq:log-likelihood} over $\t \in \cT$. We refer to such MLEs as restricted MLE. In the setting where the signal is sampled from generative models, we consider MLEs restricted to signal classes $\cT$ that are events of high probability under the respective generative model (as relevant model parameters tend to $\infty$). For more on the relationship between deterministic classes of \textit{good} signals and generative models, we direct the reader to Remark~\ref{rem:det_vs_gen}. 

Further notations and concepts used in this paper that are of a more generic nature can be found in the Appendix \ref{a:notn}.

\subsubsection{The \textit{dilute} regime: sparse collision-free signals } \label{sec:coll-free}

 
We first define the notion of collision-free property of the support of a signal, and subsequently use it for introducing the appropriate signal class for the dilute regime, which, roughly speaking, consists of signals that can at best be of size $O(L^{1/2})$ and typically of size $O(L^{1/3})$.
 
\begin{defi}
 For a vector $v \in \R^L$, we will denote by $\cD(v)$ the (multi-)set of differences $\{v(i) - v(j):  1\le i,j \le d\}$. In general, this is a set of differences with multiplicities. In case the multiplicity is exactly 1 for each difference appearing in $\cD(v)$, we call the vector $v$ \textit{collision-free}, that is there are no repeated differences in its  support. 
\end{defi} 
Notice that being collision-free  is really a property of the support $\supp(v)$ of the vector $v$. 
 
 We are now ready to define the signal class that we will investigate in the dilute regime. 
\begin{defi} \label{defi:dilute}
We consider the set $\cT  \subset \R^L$ to consist of the signals $\t:\Z_L \mapsto \R$ that satisfy the conditions outlined below.
\begin{itemize}

\item[(i)] $\t$ is collision free

\item[(ii)] There exist positive numbers $m,M,\eps>0$ (uniform for the set $\cT$) such that $m \le |\t(i)| \le M$ on $\supp(v)$, and $s:=|\supp(\t)| \ge (2+\eps) M^2/m^2$.
\end{itemize}
\end{defi}

We can then state the following theorem.
\begin{thm}\label{thm:rates-beltway}
Let $\mathcal{T}$ be the set of signals  as in Definition \ref{defi:dilute}.  Then for $\s$ bigger than a threshold $\s_0(L)$, for any signal $\t_0 \in \mathcal{T}$, the restricted MLE $\tlt_n$ for the MRA problem satisfies $\sqrt{n}{\displaystyle \varrho(\tlt_n,\theta_0)=O_p(\sigma^2)}$ as $n \to \infty$.
\end{thm}

A crucial ingredient in the proof of Theorem \ref{thm:rates-beltway} is the following curvature lower bound on the second moment tensor, which we state below as a result of independent interest.
\begin{lem} \label{lem:lbound-beltway}
Let $\mathcal{T}$ be the set of vectors $\t \in \R^L$ as in Definition \ref{defi:dilute}. Then, for any $\t,\t_0 \in \mathcal{T}$, we have
\[  \| \E_{G} [(G\t)^{\otimes 2}] -  \E_{G} [(G\t_0)^{\otimes 2}]  \|_F \ge \sqrt{\frac{2\eps}{2+\eps}} \cdot \frac{1}{\sqrt{L}} \cdot \sqrt{s} \cdot \rho(\t,\t_0)  \] for some universal constant $c$.
\end{lem}

The collision-free property of the support of the signal, as enunciated in this section, is typically associated with the signal being considerably sparse; hence the name \textit{dilute} regime. In fact, it may be shown that for the signal support to be collision-free, the size $s$ of the support cannot exceed $O(L^{1/2})$. On the other hand, it can also be shown that typical subsets (e.g., chosen uniformly at random from the co-ordinates) the collision free property holds with high probability if $s=o(L^{1/3})$. We refer the reader to Appendix \ref{a:coll_free} for further on these size bounds.

\subsubsection{The \textit{moderate} regime: generic sparse symmetric signals} \label{sec:sparse-symm}
In this section, we demonstrate that we can extend far beyond the dilute regime and obtain a sample complexity of $\sigma^4$ for generic symmetric signals in the so-called \textit{moderate} regime of sparsity, extending all the way up to support size $s=O(L/\log^5L)$.  In doing so, we invoke uncertainty principles from Fourier analysis as an effective technique for the studying MRA problem. 

To this end, we define the notion of the Bernoulli-Gaussian distribution, and the symmetric version thereof. The Bernoulli-Gaussian distribution is a popular model for modelling generic or typical sparse signals in statistics and signal processing ~\cite{Kor82,Sou11}. In this work, we use the symmetric Bernoulli-Gaussian distribution in order to model sparse symmetric signals for investigating estimation rates under the MRA model. Such symmetry hypothesis  is well-motivated by the fact that many natural objects of interest, such as molecules, exhibit symmetries that are of significance in spectroscopy \cite{bunker2006molecular,reddy1998symmetry,wigner2012group}; this includes reflection symmetries that are related to the important notion of chirality \cite{barron1986symmetry,buckingham2004chirality}.  

A signal following the Bernoulli-Gaussian distribution with variance $\zeta^2$ and sparsity $s$ consists in generating the signal support via independent random sampling of points in $\Z_L$ with probability $s/L$ each, and then independently generating the signal values on the support via a $N(0,\zeta^2)$ distribution for each point. The symmetric Bernoulli-Gaussian distribution differs from the general case defined above only in the fact that its support is constrained to be symmetric. To obtain this, we consider $\Z_L$ in its standard parametrization \eqref{eq:std-par}, and pick the \textit{positive} part $A_+$ of the support by independent random sampling from $\Z_L^+$ with probability $s/L$, and then obtain the full symmetric support $A$ via reflection about the origin, i.e. $A=A_+ \cup (-A_+)$.

While the Bernoulli-Gaussian distribution is standard for modelling sparse signals, our results apply to far more general signal classes. This includes the $N_{[-s,s]}^{\symm}(0,\zeta^2 I)$ distribution, which entails that the support is $[-s,s]$ and the signal values are independent $N(0,\zeta^2)$ random variable. In fact, other than independent Gaussian values, our results only require that the signal support be sparse and sufficiently \textit{generic}, in a precise arithmetic sense that we call \textit{cosine genericity}. 

We call such signals \textit{generic sparse symmetric signals}. Our main estimation rate results will be stated below in terms of this signal class; the precise and detailed definitions for it are provided in Appendix \ref{a:gen_signal}. 

\begin{thm} \label{thm:sparse-symm-G}
Let $\log^9 L \le s \le L/\log^5 L$. Consider a generic $s$-sparse symmetric signal $\t_0: \Z_L \mapsto \R$, with dispersion $\zeta^2$, sparsity constants $(\a,\b)$ and index $\tau>0$.
Then for $\s$ bigger than a threshold $\s_0(L)$, with high probability in $\t_0$,  the restricted MLE $\tlt_n$ for the MRA problem satisfies ${\displaystyle \sqrt{n}{\displaystyle \varrho(\tlt_n,\theta_0)=O_p(\sigma^2)}}$ as $n \to \infty$.
\end{thm}

Theorem \ref{thm:sparse-symm-G} enables us to deduce an improved sample complexity of order $\sigma^4$ for signals sampled from the symmetric Bernoulli-Gaussian distribution.
\begin{cor} \label{cor:sparse-symm-BG}
Let $\log^9 L \le s \le L/\log^5 L$. Consider a signal $\t_0$ sampled from either of:
\begin{itemize}
\item[(i)]  The symmetrized Bernoulli-Gaussian distribution on $\Z_L$ with mean zero, variance $\zeta^2$ and sparsity $s$, or
\item[(ii)]   The  $N_{[-s,s]}^{\mathrm{symm}}(0,\zeta^2I)$ distribution.
\end{itemize}
Then, for $\s$ bigger than a threshold $\s_0(L)$, with high probability in $\t_0$, the restricted MLE  $\tlt_n$ for the MRA problem satisfies 
$\sqrt{n}{\displaystyle \varrho(\tlt_n,\theta_0)=O_p(\sigma^2)}$ as $n \to \infty$.
\end{cor}

\subsubsection{Dependence on sparsity and ambient dimension} \label{sec:sL-dependence}
In important signal classes, we can obtain the dependence of asymptotic estimation rates on the parameters $(s,L)$. We record them in the following result.

 
\begin{thm} \label{cor:dependence}
Let  $\t_0$ be the signal in the MRA model, and $\s$ bigger than a threshold $\s_0(L)$.
\begin{itemize}
\item[(i)]  If $\t_0$ is as in Definition \ref{defi:dilute} with $|\supp(\t_0)|=s$, then we have the $\lim_{n \to \infty}\sqrt{n}\varrho(\tlt_n,\t_0)/\sigma^2=O_p(1)$ as a function of $s,L$.
\item[(ii)]   If $\t_0$ is sampled from the symmetric Bernoulli-Gaussian distribution with mean 0, variance $\zeta^2$ and sparsity parameter $s$ with $\log^9 L \le s \le L/\log^5 L$, then with high probability in $\t_0$, we have $\lim_{n \to \infty} \sqrt{n} \varrho(\tlt_n,\t_0)/\sigma^2=O_p(s^{3.5})$ as a function of $s,L$.  
\end{itemize}
\end{thm}

We emphasize here that our rate bounds are asymptotic in $n$, and therefore necessarily \textit{local} in character, focusing on a small enough neighbourhood of the signal (that will generally depend on $L$). In contrast,  non-asymptotic bounds that are generally \textit{global} over the set of allowable signals, and therefore will usually exhibit an $L$ dependence, such as the standard $\sqrt{\log L}$ dependence in much of the signal processing literature. On a related note, in this work we concern ourselves with the leading, dominant term in an asymptotic expansion of $\sqrt{n}\varrho(\tlt_n,\t_0)/\sigma^2$ in the regime of large $L$; higher order terms in this expansion will generically depend on $L$. 

We observe that the asymptotic upper bound in the dilute regime is $O_p(1)$, but in the regime of moderate sparsity it is growing as $s^{3.5}$. This is perhaps to be expected, in tune with the fact that a $\sigma^2$ dependence of the estimation rate eventually breaks down for signals with full support. 

It is of interest to make explicit the role of the ambient dimension $L$ as a quantifier in the main results of this paper. The main results, such as Theorems \ref{thm:sparse-symm-G}, \ref{cor:dependence} and \ref{thm:rates-beltway} and Corollary \ref{cor:sparse-symm-BG},  entail statements regarding generic signals. This genericity refers to the fact that the statements of these results hold for a set of signals that, under suitable distributions on the signal space (whose specifics are clarified for each result), has probability at least $1-p(L)$, where $p(L) \to 0$ as $L$ tends to infinity. While a precise bookkeeping of our arguments would indeed yield explicit formulae for such sequences $p(L)$ for the relevant distributions considered in this paper, we prefer not to pursue that route so as to maintain brevity, given that the statements in their present forms already capture the main qualitative phenomena and key dependencies.

Thus, our results are in particular not asymptotic in $L$: indeed, the results hold for each $L$ (bigger than some threshold $L_0$) for a class of signals $\mathcal{S}(L)$ that depends on $L$. As $L$ grows, the  probability measure of $\mathcal{S}$ under a natural distribution converges to 1. However, the results do have a clear interpretation even without letting $L \to \infty$, which is the point of view we take in this paper. In the recent work \cite{romanov2021multi}, the authors undertake an examination of the MRA problem in the \textit{high dimensional} regime, where $L,\s \to \infty$ jointly with the special parametric dependence $\s = L/\a \log L$ for a parameter $\a>0$. In such a situation, the order of the limits $n,L \to \infty$ becomes important. In the present work, we interpret the results as above -- with fixed $L$ and $n \to \infty$. Thus, the results in this paper are of a different flavour and indeed not comparable to the high dimensional scenario as in \cite{romanov2021multi}. A synthesis of these two points of views, however, remains an interesting question for future work.

\subsubsection{A technical lemma} \label{sec:tech_lem}

A useful ingredient in the proof of Theorem \ref{thm:sparse-symm-G} is a deterministic technical  lemma which we state below.

\begin{lem} \label{lem:symm-tech}
Consider the set of signals $\cT$ on $\Z_L$ in the standard enumeration \eqref{eq:std-par}, defined as follows.
For each $\t \in \cT$, we have
\begin{itemize}
\item[(i)] $\t$ is symmetric, i.e. $\t(i)=\t(-i) \forall i \in \Z_L$.
\item[(ii)] There exist $m_{\cT},M_{\cT}>0$, uniform in $\t$, such that $m_{\cT} \le |\t(i)| \le M_{\cT}$ on $\supp(\t)$.
\item[(iii)] There exists $\L \subset \Z_L$, possibly depending on $\t$, such that:
           \begin{itemize}
           \item[(a)] $$c_1 \cdot \frac{1}{L} \|h\|_2^2 \le \frac{1}{|\L|}\sum_{\xi \in \L} |\hh(\xi)|^2 \le c_2 \cdot \frac{1}{L} \|h\|_2^2 $$ for all $h$ with $\supp(h) \subseteq \supp(\t)$, with positive constants $c_1,c_2$ uniform in $h,\t$.
           \item[(b)] $\min_{\xi \in \L}|\hht(\xi)| \ge \fm(\cT)$, for some $\fm(\cT)>0$ that is uniform in $\t$.
           \end{itemize}
\end{itemize}
Then, for $\s$ bigger than a threshold $\s_0(L)$ and any signal $\t_0 \in \cT$, the restricted MLE $\tlt_n$ satisfies $\sqrt{n}{\displaystyle \varrho(\hht_n,\theta_0)=O_p(\sigma^2)}$  as $n \to \infty$.
\end{lem}

\begin{rem} \label{rem:only_sparsity}
A significant setting in the context of the conditions (i)-(iii) above is when the signal class $\cT$ is sparse with typical support size $s$, and $m_{\cT},M_{\cT},\fm(\cT)$ as well as the constants $c_1,c_2$ all depend only on the sparsity $s$ and not on $L$. This is, in fact, true for both the sparse symmetric Bernoulli-Gaussian distribution and the $N_{[-s,s]}^{\mathrm{symm}}(0,\zeta^2I)$ distributions that we consider in this paper, and lead to estimation rates that depend asymptotically only on the sparsity and not on the ambient dimension. 
\end{rem}

\begin{rem} \label{rem:det_vs_gen}
We observe that $\cT$ in Lemma \ref{lem:symm-tech} is a deterministic subset of the space of signals, whereas  Theorem \ref{thm:sparse-symm-G} and Corollary \ref{cor:sparse-symm-BG} consider a typical signal from certain generative models. We can, however, easily reconcile the two by considering a large, compact set $\cT$ of signals that carries a high probability measure under the respective generative distribution and satisfies the conditions of Lemma \ref{lem:symm-tech}. This allows us to deduce improved rates of estimation for \textit{typical} signals under generative distributions by making use of the deterministic Lemma \ref{lem:symm-tech}.
\end{rem}

\subsubsection{Estimation rates beyond sparsity} \label{sec:slow-support}
Theorems \ref{thm:rates-beltway} and \ref{thm:sparse-symm-G} establishes a sample complexity of order $\sigma^4$ for signals with sparse support. Complementary to Theorem \ref{thm:rates-beltway}, in this section we examine classes of signals with very different structural properties compared to sparse support, and show that in such settings, ${\sigma^2}/{\sqrt{n}}$ rates of estimation is generically impossible. 

\begin{thm} \label{thm:impossibility}
Let $\mathcal{T}$ be the set of vectors $\t \in \R^L$ is such that its support $\supp(\t)$ is all of $\Z_L$ and $m \le |\t(i)| \le M$ on $\supp(v)$. Then, for any signal $\t_0 \in \mathcal{T}$ the restricted MLE $\hht_n$ satisfies ${\displaystyle \sqrt{n} \varrho(\hht_n,\theta_0)=\Omega_p(\sigma^3)}$ as $n \to \infty$.
\end{thm}
This will follow from the following Lemma, to state which we need to introduce the following notation. For a signal $\phi=(\phi^{(1)},\ldots,\phi^{(L)}) \in \R^L$, we define the average
\[ \bp:=\frac{1}{L} \sum_{i=1}^L \phi^{(i)} \] 
\begin{lem} \label{lem:lbound-impossibility}
Let $\mathcal{T}$ be the set of vectors $\t \in \R^L$ is such that its support $\supp(\t)$ is all of $\Z_L$ and $m \le |\t(i)| \le M$ on $\supp(v)$. Then, there exist a sequence  $\{\t_k\}_{k >0} \subset \mathcal{T}$ such that $\rho(\t_k,\t_0) \to 0$ as $k \to \infty$, $\bt_k=\bt_0$, and we have
\begin{equation}  \label{eq:lb-imp}
\| \E_{G} [(G\t_k)^{\otimes 2}] -  \E_{G} [(G\t_0)^{\otimes 2}]  \|_F \le C \varrho(\t_k,\t_0)^2. 
\end{equation}
\end{lem}
We may compare the statement of Lemma \ref{lem:lbound-impossibility} with that of Lemma \ref{lem:lbound-beltway}, and notice that the lower bound in Lemma \ref{lem:lbound-beltway} is linear in $\varrho(\t,\t_0)$; whereas the upper bound in Lemma \ref{lem:lbound-impossibility} is quadratic in $\varrho(\t,\t_0)$.

In summary, attaining improved rates of estimation via the second moment is not possible when the signal has full support.

\subsection{Main ideas and ingredients} \label{sec:ideas}
Our investigation of the asymptotic estimation rates for the MRA model connects to several classical topics in applied mathematics, including  the beltway problem related to combinatorial optimization and uniform uncertainty principles from harmonic analysis.


\subsubsection{The beltway problem}
The beltway problem consists in recovering a set $\cS$ of numbers from their pairwise differences $\cD$, up to the trivial symmetry of translating all the numbers in the set by the same amount. It is closely related the so-called \textit{turnpike problem} or the \textit{partial digest problem}, and is of interest in computational biology, where it arises naturally in DNA restriction site analysis among other problems.
A set of integers is said to be \emph{collision-free} if all the pairwise distances obtained from that set are distinct.
In 1939, Piccard ~\cite{Pic39} conjectured that, if two sets $\cS_1$ and $\cS_2$ of integers have the same set of pairwise differences $\cD$, and the pairwise differences are known to be unique (i.e., $\cS_1$ and $\cS_2$ are collision-free), then the sets $\cS_1$ and $\cS_2$ must be translates of each other.

Following major advances by Bloom ~\cite{Blo77}, a description of the complete landscape of Piccard's conjecture was obtained by Bekir and Golomb ~\cite{BekGol07,bekir2004nonexistence}, who demonstrated in particular that the conjecture is true for all sets of cardinality greater than 6.

The upshot of these developments  is that if $\cS$ is a set of integers with $|\cS| \ge 7$ and such that the pairwise distances of the numbers in $\cS$ are distinct (in other words, $\cS$ is collision-free), then the set $\cS$ is uniquely determined (up to translations) by its pairwise distances. This will be exploited in our investigations of the MRA estimation rates. In particular, the beltway problem motivates our definition of the ``collision-free'' condition on the support of the signal, which will be used for obtaining improved estimation rates in the dilute regime of sparsity.

\subsubsection{Uniform Uncertainty Principles:}
Uncertainty principles have a long history in harmonic analysis and in applied mathematics, starting from Heisenberg's celebrated Uncertainty Principle in quantum mechanics \cite{FolSit97,Fef83}. Roughly speaking, these entail that a function cannot both be simultaneously localized (i.e., have a `small support' in an appropriate sense) both in the physical space and the Fourier space. This would imply that for a function with a small support in the physical space (e.g., a sparse signal), the Fourier transform would be overwhelmingly non-vanishing, and therefore we would need almost all of the Fourier coefficients in order to fully capture the `energy' (i.e. the $L^2$ norm) of the function, by the Parseval-Plancherel Theorem.

However, if our goal is to approximate the function (up to a limited multiplicative error in the total energy), it may actually suffice to work with a relatively small subset of Fourier coefficients. In fact, such an appropriate set of Fourier coefficients may be obtained via random sampling, and furthermore, such a `good' set of frequencies may be shown to provide a good approximation simultaneously for all sparse signals. Results in this vein are referred to as \emph{uniform uncertainty principles}; for an expository account we refer the reader to ~\cite{Tao08} (Chap. 3.2). These are also closely related to the so-called \emph{Restricted Isometry Property} (RIP) for random sub-sampling of Fourier matrices (c.f., \cite{RudVer08}).


\subsection{Connection to Phase Retrieval} \label{sec:Phase_Retrieval}

\subsubsection{Generalities} \label{sec:PR-gen}
Phase retrieval is a central and long-standing question in applied mathematics that find applications in a variety of domains such as astronomy, electron microscopy and optical imaging, and has emerged in recent years as a widely studied question in the field of signal processing ~\cite{She15,Can15,Fie13}.

A key connection between the MRA model and the phase retrieval problem arises via second moment tensors. Theorem \ref{thm:BRW2} (Theorem 9, \cite{BanRigWee17}), as well as the related work \cite{PerWeeBan19}, establishes a clear connection between the $O(\sigma^4)$ sample complexity in the MRA problem and being able to recover the true signal from (estimates of) its second moment tensor,  via an expansion of the Kullback-Leibler divergence in terms of moment tensors. The second moment tensor of a signal $\t$ is a circulant matrix related to the convolution $\t \star \ct$, whose Fourier transform $\widehat{\t \star \ct} = |\hht|^2$ as functions. 

The problem of (Fourier) phase retrieval entails signal recovery from the modulus of its random linear measurements. This problem has been well-investigated in  recent years, as indicated by a substantial literature ~\cite{bendory2017fourier,ohlsson2014conditions,beinert2017sparse,jaganathan2012recovery,jaganathan2017sparse,ranieri2013phase}. 

The recent work \cite{bendory2020toward} examines the question of recovering a sparse signal from its power spectrum in the context of crystallographic  phase retrieval in a non-randomized setting. The connection of such questions to the turnpike problem was considered in the earlier work \cite{ranieri2013phase}. The paper \cite{romanov2021multi} considers the sample complexity of MRA in high dimensions (under a Gaussian prior), exploring in particular the interplay between the dimension and the noise level.   
 
However, as noted earlier, our observational setting in the MRA model with the latent group action has a very different and more complicated structural setup than the existing literature on phase retrieval, which largely focuses on a specific setting akin to compressed sensing with limited information.

The present work focusses on statistical rates of estimation, leaving aside the question of algorithmic implementation for future work. The elaborate literature on phase retrieval, on the other hand, makes a detailed exploration of algorithmic issues, which might be of natural interest in this regard. 

\subsubsection{Crystallographic Phase Retrieval}

The results and methods in the present work have applications to phase retrieval, in particular the problem of crystallographic phase retrieval. The latter problem is believed to be perhaps the most important phase retrieval setup, where the interest is in recovering a sparse signal from its \textit{power spectrum}, i.e., the magnitude of its Fourier transform at various frequencies   \cite{elser2018benchmark,bendory2020toward}. This is equivalent to recovering a signal from its \textit{periodic auto-correlation}. The problem is motivated by X-ray crystallography, a technique for determining molecular structure \cite{millane1990phase}. 

Formally, let $\t_0$ be an $s$-sparse signal in ambient dimension $L$. In crystallographic phase retrieval, we are interested in recovering $\t_0$ from the magnitude of its discrete Fourier transform, namely $\{|\hht_0(j)|^2\}$ at measurement frequency $j$. Equivalently, the interest is in recovering $\t_0$ from its periodic auto correlations, which are given by \[ \mathcal{A}_{\t_0}(l) = \sum_{i \in \Z_L} \t_0(i) \t_0(i+l \; \text{mod} L) \quad \text{for} \; i \in \Z_L. \]
Clearly, the power spectrum of a signal remains invariant under rotations and reflections, and objective is to recover the signal upto these \textit{intrinsic} symmetries.

The theoretical understanding of crystallographic phase retrieval is rather limited (c.f. \cite{ranieri2013phase,bendory2020toward}). In particular, even the fundamental question of uniqueness is poorly understood.   
In \cite{bendory2020toward}, conjectures were laid out regarding the uniqueness of sparse signal recovery from its power spectrum, that predicted in particular that under general conditions unique recovery should be possible when $s/L \le 1/2$.

Our investigations in this paper have implications for the crystallographic phase retrieval problem, indicating local uniqueness of signal recovery from power spectrum for sparse collision-free signals (i.e., in the dilute regime of sparsity) and for generic sparse symmetric signals with $s=O(L/\log^5 L)$ (i.e., in the moderate regime of sparsity). 

This follows from the fact that our $\sqrt{n} \varrho(\tlt_n,\t_0) = O_P(\s^2)$ estimation rates for such signals are a consequence of lower bounds on the second moment difference tensors, such as in Lemma~\ref{lem:lbound-beltway} and  \eqref{eq:curv_LB_moderate}. The second moment tensor $\E_{\cG}[ (g \cdot \t)^{\otimes 2} ]$ is a  matrix whose entries are precisely the periodic auto correlations $\mathcal{A}_\t$, so lower bounds such as Lemma \ref{lem:lbound-beltway} and  \eqref{eq:curv_LB_moderate} indeed demonstrate unique recovery of $\t$ from $\mathcal{A}_\t$. The uniqueness is \textit{local}, because our estimation rates are local in character, which entails that lower bounds such as \eqref{eq:curv_LB_moderate} hold in a neighbourhood of the true signal $\t_0$. Further, the uniqueness is among a class of signals that share similar sparsity features as the true signal, e.g. the signal class  $\cT$ in Lemma~\ref{lem:lbound-beltway}. Application of the techniques of the present work to obtain more extensive uniqueness guarantees for the crystallographic phase retrieval problem is an interesting problem for future research.

%
%
%

\section{Rates of estimation and Curvature of the KL divergence} \label{sec:Curv}
\subsection{Estimation Rates and Curvature} \label{sec:est_rates_curv}
In this work, we establish, under very general conditions, quadratic rates  of estimation (i.e., scaling as $\sigma^2$) in the MRA problem, in the context of Theorem \ref{thm:lower} (and \eqref{eqn:minimax_lower} in particular).

Our point of departure is the population risk of the MRA model, given by
\begin{equation}
R(\t)=-\E_{p_{\t_0}}[\log p_\t(Y)] + C,
\end{equation}
where $C$ is a universal constant. Clearly, we have
\begin{align*}
R(\t)= &-\int \log p_\t(y) p_{\t_0}(y) \ud y + C  \\ = &\int \log \left( \frac{p_{\t_0}(y)}{p_{\t}(y)} \cdot \frac{1}{p_{\t_0}(y)} \right) p_{\t_0}(y)\ud y + C \\ = &D_{KL}(p_{\t_0} || p_{\t} ) - \left(\int p_{\t_0}(y) \log p_{\t_0}(y) \ud y \right) + C
\end{align*}
where $D_{KL}(p_{\t_0} || p_{\t} )$ is the Kullback-Leibler divergence between $p_{\t_0}$ and $p_{\t}$. Since $\t_0$ is fixed, as a function of $\t$, the population risk $R(\t)$ equals
\begin{equation}
R(\t)=D_{KL}(p_{\t_0} || p_{\t} ) + C(\t_0),
\end{equation}
where $C(\t_0)$ is a function of $\t_0$.

The Fisher information matrix of the MRA model is given by
\begin{equation} \label{eq:Fisher}
I(\t_0)= -\E[\nabla_\t^2 \log p_\t(Y)\big|_{\t=\t_0}]=\nabla_\t^2 R(\t_0),
\end{equation}
where $\nabla_\t^2$ denotes the Hessian with respect to the variable $\t$. It has been demonstrated  \cite{ABL18} that the MLE $\tlt_n$ is an asymptotically consistent estimate for the true signal $\t_0$ in the MRA model. This immediately enables us to invoke standard asymptotic normality theory for maximum likelihood estimators and conclude that: 
\begin{equation} \label{eq:asymp_normality}
\sqrt{n}(\tlt-\t_0) \quad \text{is asymptotically normal with mean 0 and variance} \quad I(\t_0)^{-1}. 
\end{equation}
From the considerations above, we may conclude that the asymptotic covariance is given by $\left(\nabla_\t^2 D_{KL}(p_{\t_0}||p_\t)\right)^{-1}$. For a detailed discussion on such asymptotic normality, we refer the reader to \cite{VdV00}, in particular Sections 5.3 and 5.5 therein.

We observe that the probability distribution $p_\t$ as well as  $D_{KL}(p_\t \| p_\phi)$ are invariant under the action of $\cG$, i.e., invariant under the transformations $\t \mapsto G \cdot \t$ for $G \in \cG$. As a result, for $\varrho(\t,\t_0)$ small enough (equivalently, $\|\t-\t_0\|_2$ small enough), we may assume without loss of generality that $\varrho(\t,\t_0)= \frac{1}{\sqrt{L}}\|\t-\t_0\|_2$ (c.f., \cite{BanRigWee17}; esp. the proof of Theorem 4 therein). Since $\|\tlt_n-\t_0\|_2 \to 0$ as $n \to \infty$, this will be true for $\varrho(\tlt_n,\t_0)$ with high probability.  

The upshot of the asymptotic normality discussed above is that, as $n \to \infty$, the quantity $\rho(\tlt_n,\t_0)$ (which equals $\frac{1}{\sqrt{L}}\|\tlt_n-\t_0\|_2$ with high probability), is  of the order  \[ n^{-1/2} \sqrt{\mathrm{Tr}\left[\frac{1}{L} \cdot I(\t)^{-1} \right]}= n^{-1/2} \sqrt{\mathrm{Tr}\left[\frac{1}{L} \cdot \nabla_\t^2 D_{KL}(p_{\t_0}||p_\t)^{-1} \right]},\]
where $\mathrm{Tr}[\cdot]$ denotes the trace. This is related to the fact that if $\mathbb{X} \sim N(0,\Sigma)$, then $\E[\|\mathbb{X}\|_2^2]=\E[\mathrm{Tr}(\mathbb{X}^* \mathbb{X})]=\mathrm{Tr}(\Sigma)$. 

Thus, the estimation rate for the MRA problem asymptotically depends on $\sigma$ via the dependence of $\sqrt{\mathrm{Tr}[\frac{1}{L} \cdot \nabla_\t^2 D_{KL}(p_{\t_0}||p_\t)^{-1}]}$ on $\sigma$. In view of this, curvature  bounds on $D_{KL}(p_{\t_0}||p_\t)$ assume significance. We record this in the following proposition.

To this end, we recall the metric $\rho(\cdot,\cdot)$ \eqref{eq:def_rho}, which is essentially a scaling of $\varrho$: indeed, $\rho(\cdot,\cdot) = \sqrt{L} \varrho(\cdot,\cdot)$.

\begin{prop} \label{prop:CurveLB-normality}
We have the following relations between curvature bounds on $D_{KL}$ and the asymptotic behaviour of $\varrho(\tlt_n,\theta_0)$:
\begin{itemize}
\item[(i)] If $D_{KL}(p_{\t_0}||p_\t) \ge K_1(\sigma) \rho(\t,\t_0)^2$ for $\t$ in a neighbourhood $U$ of $\t_0$, then $\sqrt{n} \varrho(\tlt_n,\theta_0)=O_p \left(\frac{K_1(\sigma)^{-1/2}}{\sqrt{L}} \right)$.
\item[(ii)] If $D_{KL}(p_{\t_0}||p_\t) \le K_2(\sigma) \rho(\t,\t_0)^2$ for $\t$ in a neighbourhood $U$ of $\t_0$, then $\sqrt{n} \varrho(\tlt_n,\theta_0)=\Omega_p\left(\frac{K_2(\sigma)^{-1/2}}{\sqrt{L}} \right)$.
\end{itemize}
\end{prop}

We defer the proof of Proposition \ref{prop:CurveLB-normality} to Section \ref{sec:curv_prop}.

Curvature  bounds as in Proposition \ref{prop:CurveLB-normality} would, in particular, be implied by upper and lower bounds on $D_{KL}(p_{\t_0}||p_\t)$ in the form of $K(\sigma) \|\t-\t_0\|^2$ (valid on some neighbourhood $U$ of $\t_0$) -- in which case, the asymptotic estimation rate in the MRA problem would scale as $\frac{K(\sigma)^{-1/2}}{\sqrt{L}}\cdot \frac{1}{\sqrt{n}}$.

We introduce the difference of the $m$-th moment tensors corresponding to two signals $\t,\phi$:
\[\Delta_m(\t,\phi):= \E[(G\t)^{\otimes m}] -   \E[(G\phi)^{\otimes m}]. \] Furthermore, by the (Frobenius)  norm $\|\cdot\|$ for a tensor, we will denote its Hilbert Schmidt norm. In what follows, we will invoke two results from \cite{BanRigWee17}, in which we will make use of the distance $\rho$ (c.f. \ref{eq:def_rho}).

This allows us to state the following results from \cite{BanRigWee17}.
\begin{lem} \label{lem:BRW2}~\cite{BanRigWee17}[Lemma 8]
If $\tilde{\theta}=\t-\E_\cG[G\t]$ and $\tilde{\phi}=\phi-\E_\cG[G\phi]$, then 
\[D_{KL}(p_\t || p_\phi) = D_{KL}(p_{\tilde{\theta}} || p_{\tilde{\phi}}) + \frac{1}{2\s^2} \|\Delta_1(\t,\phi)\|^2.  \]
\end{lem}

\begin{thm} \label{thm:BRW2}~\cite{BanRigWee17}[Theorem 9]
Let $\t,\phi \in \R^L$ satisfy $3\rho(\t,\phi) \le \| \t \| \le \sigma$ and $\E[G\t]=\E[G\phi]=0$. Let $\Delta_m = \Delta_m(\t,\phi)= \E[(G\t)^{\otimes m}] -   \E[(G\phi)^{\otimes m}]$. For any $k \ge 1$, there exist universal constants $\underline{C}$ and $\overline{C}$ such that
\[  \underline{C}  \sum_{m=1}^{\infty} \frac{\|\Delta_m\|^2}{(\sqrt{3} \sigma)^{2m}m!} \le D_{KL}(p_\t || p_\phi)  \le  2 \sum_{m=1}^{k-1} \frac{\|\Delta_m\|^2}{\sigma^{2m}m!} + \overline{C} \frac{\|\t\|^{2k-2}\rho(\t,\phi)^2}{\sigma^{2k}}. \]
\end{thm}

We now use Lemma \ref{lem:BRW2} and Theorem \ref{thm:BRW2} in order to obtain bounds on $D_{KL}(p_\t \| p_\phi)$ that are tailored to our specific requirements in the present paper, focussing mostly on the second moment difference tensor in the context of Theorem \ref{thm:BRW2}. 

To this end, we consider the following results. For notational simplicity, we will use the notation $D_{KL}(\t_1 \| \t_2)$ to denote $D_{KL}(p_{\t_1} \| p_{\t_2})$. Further, for any $\t  = (\t^{(1)},\ldots,\t^{(L)}) \in \R^L$, we denote $\bt=\frac{1}{L}\sum_{i=1}^L \t^{(i)}$.
\begin{prop} \label{prop:linear_LB}
Let $\t,\phi \in \cT \subset \R^L$  belong to a bounded set $\cT$ of signals. Then for $\s$ bigger than a threshold $\s_0(L)$, and  $\varrho(\t,\phi)$ small enough, we have
\[ D_{KL}(\t || \phi) \ge C \s^{-4} \cdot  \|\Del_2(\t,\phi)\|_F^2. \]
\end{prop}

\begin{prop} \label{prop:quadratic_UB}
Let $\t,\phi \in \cT \subset \R^L$  belong to a bounded set $\cT$ of signals, such that $\bt=\bp$ and $\|\Del_2(\t,\phi)\|_F \le c \rho(\t,\phi)^2$. Then, for  $\varrho(\t,\phi)$ small enough we have 
\[
 D_{KL}(\t || \phi)  \le C \frac{\|\t\|_2^4}{\sigma^6} \cdot \rho(\t,\phi)^2
\]
for some positive number $C$.
\end{prop}

We defer the proof of Propositions \ref{prop:linear_LB} and \ref{prop:quadratic_UB} to Section \ref{sec:curv_prop}.


\section{The dilute regime of sparsity and the Beltway problem}
\label{sec:rates-beltway}

\subsection{Proof of Theorem \ref{thm:rates-beltway}}
In this section, we will establish Theorem \ref{thm:rates-beltway} and Lemma \ref{lem:lbound-beltway}.

\begin{proof}[Proof of Theorem \ref{thm:rates-beltway}]
We combine Lemma \ref{lem:lbound-beltway} with Proposition \ref{prop:linear_LB} to deduce that, for $\s \ge \s_0(L)$ and any $\t,\t_0 \in \cT$ such that $\varrho(\t,\t_0)$ is small enough, the Kullback-Leibler divergence \[D_{KL}(\t||\t_0) \ge c \s^{-4} {\frac{s}{L}} \cdot \rho(\t,\phi)^2\] for some positive number $c$. This enables us to invoke Proposition \ref{prop:CurveLB-normality} part (i) and deduce the desired asymptotic rate of estimation. This completes the proof.
\end{proof}

As is evident from the proof of Theorem \ref{thm:rates-beltway}, the key phenomenon to understand in the setting of the present section is the curvature lower bound as encapsulated in Lemma \ref{lem:lbound-beltway}. We now proceed to the proof of this important result.

\begin{proof}[Proof of Lemma \ref{lem:lbound-beltway}]
As in the statement of the lemma, we focus on the situation where $\varrho(\t,\t_0)$ is small, and we recall that, for $\vr(\t,\t_0)$ small enough, we may take $\vr(\t,\t_0)=\|\t-\t_0\|_2$ because of the $\cG$ invariance of $D_{KL}$ and the moment tensors $\Delta_m$. This is the setting in which we will work. 

Notice that in this setting, $\t$ and $\t_0$ have the same support. This follows from the assumption on our signal class $\cT$ that for any $u \in \cT$, we have $|u(i)| \ge m \quad \forall i \in \supp(u)$. As a result, if $i \in \supp(\t) \Delta \supp(\t_0)$, then $|\t(i)-\t_0(i)|\ge m$, which implies that $\vr(\t,\t_0)=\|\t-\t_0\|_2 \ge m$, which contradicts the smallness of $\vr(\t,\t_0)$.

We  set $h=\t-\t_0$, to be thought of has having small $L^2$-norm. Notice that the above discussion implies $\supp(h) \subseteq \supp(\t_0)$. We  consider \[\|\E_\cG[G\t^{\otimes 2}G^*] - \E_\cG[G\t_0^{\otimes 2}G^*]\|_F=\|\E_\cG[G(\t_0+h)^{\otimes 2}G^*] - \E_\cG[G\t_0^{\otimes 2}G^*]\|_F.\] To the leading order in $h$, this is $\|\E_\cG[\t_0 h^* + h \t_0^*]\|_F$, where $^*$ denotes  transpose. Since $\|h\|_2$ is small, it suffices to consider this leading order term, and demonstrate that this is $\ge c \sqrt{s} \|h\|_2$. Henceforth, we focus on this objective.

We then have
\[ \bigg( \E_\cG[G(\t_0 h^* + h \t_0^*)G^*]  \bigg)_{i,j} =   \frac{1}{L}\sum_{g \in \Z_L} [\t_0(i+g)h(j+g) + h(i+g)\t_0(j+g)] . \] In our subsequent considerations, we will use the symbol $J$ to denote the matrix $\E_\cG[G(\t_0 h^* + h \t_0^*)G^*] $.
Observe that $J$ is a Toeplitz matrix. We can therefore denote the entries of $J$ as $J_{i,j}=J_{j-i}$.
Furthermore,  for each $i,j$ we have $J_{i,j}=J_{i+k,j+k}$, for any $k \in \Z_L$ and the sums $i+k,j+k$ in the indices being interpreted to be sums in $\Z_L$. In view of this, we can write $\|J\|_F^2=L\sum_{k=0}^{L-1}|J_{k}|^2$. From here on, we will focus on estimating from below the sum
$\sum_{k=0}^{L-1}|J_{k}|^2$.

Note that $\t_0(i+g)h(j+g)$ or $h(i+g)\t_0(j+g)$ is non-zero only if both $i+g$ and $j+g$ belong to the support of $\t_0$ (which contains the support of $h$). But since all non-zero differences occur exactly once (collision-free property), there exists a unique $g=g(i,j)$ such that $ [\t_0(i+g)h(j+g) + h(i+g)\t_0(j+g)]$ can possibly be non-zero. In particular, this means that, for $J_k$ to be non-zero, $k$ has to belong to $\cD(\t_0)$.

Suppose $0 \ne k \in \cD(\t_0)$ and let $i,j \in \supp(\t_0)$ be such that $j-i=k$. By the collision-free property, there is exactly one such pair $(i,j)$. Therefore \[|J_k|^2=\frac{1}{L^2}[\t_0(i)h(j) + h(i)\t_0(j)]^2.\] Conversely, if $i \ne j$ are such that $j-i \in \cD(\t_0)$, then  there is a (unique) contribution to the sum $\sum_k|J_{k}|^2$ by an amount ${\displaystyle
\frac{1}{L^2}[\t_0(i)h(j) + h(i)\t_0(j)]^2
\displaystyle}$.

Finally, note that in case either $i$ or $j$ does not belong to $\supp(\t_0)$, we have \[\frac{1}{L^2}[\t_0(i)h(j) + h(i)\t_0(j)]^2=0.\]

Putting together all of the above, and denoting $s:=|\supp(\t_0)|$ we have
\begin{align*}
&\|J\|_F^2  \quad  =  L \sum_{k=0}^{L-1}|J_{k}|^2  \\  \ge &  L\sum_{k=1}^{L-1}|J_{k}|^2 \quad  \ge    \frac{1}{L} \cdot \sum_{i \ne j} [\t_0(i)h(j) + h(i)\t_0(j)]^2 \\
= & \frac{1}{L} \cdot \sum_{i \ne j} [\t_0(i)^2h(j)^2 + h(i)^2\t_0(j)^2 + 2 \t_0(i)h(i) \t_0(j)h(j)]. \\
= & \frac{1}{L} \cdot \left[2 \left( \sum_i \t_0(i)^2 \right)  \left( \sum_j h(j)^2 \right)  - 2\sum_i \t_0(i)^2 h(i)^2 + 2 \left(\sum_{i} \t_0(i) h(i) \right)^2 \right.  \\ 
& \left.  \qquad \qquad - 2 \sum_i \t_0(i)^2 h(i)^2 \right] \\
\ge &  \frac{2}{L} \cdot \left[ \|\t_0\|^2 \|h\|^2 -  2 \sum_i \t_0(i)^2 h(i)^2 \right] 
\ge   \frac{2}{L} \cdot \left[ \|\t_0\|^2 \|h\|^2 -  2 M^2\sum_i h(i)^2 \right] \\
\ge &   \frac{2}{L} \cdot  \left(\|\t_0\|^2 - 2M^2\right)\|h\|^2   
\ge   \frac{2}{L} \cdot  \left(sm^2 - 2M^2 \right)\|h\|^2   \\
= & \frac{2s}{L}(m^2-\frac{2M^2}{s})\|h\|^2, \numberthis \label{eq:lower_bnd} \\
\end{align*}
where, in the last few steps, we have used the fact that $m \le |\t_0(i)| \le M ~ \forall i \in \supp(\t_0)$. 

For $s \ge (2+\eps) M^2/m^2$ with $\eps>0$, the lower bound in \eqref{eq:lower_bnd} can be further bounded below by $\frac{2 \eps}{2 + \eps} \cdot \frac{s}{L} \cdot \|h\|^2$, as desired.
\end{proof}

\section{Curvature of KL divergence and the Fourier Transform}
In this section, we will show that, without additional structural assumptions on the signal (such as sparsity), the second moment is \textit{generically} insufficient to achieve $O(\s^4)$ sample complexity, as indicated in Theorem \ref{thm:impossibility}. In doing so, we will study the MRA problem in general, and the second moment tensor in particular, from the point of view of the Fourier transform of the signal $\t$.

\begin{proof}[Proof of Theorem \ref{thm:impossibility}]
We begin with the fact that  and that $D_{KL}(p_{\t_0}\|p_{\t})$ has a local minimum at $\t=\t_0$,  
and on a related note, we have 
\[ D_{KL}(p_{\t_0}\|p_{\t})\big|_{\t=\t_0}=0,  \quad \nabla_\t D_{KL}(p_{\t_0}\|p_{\t})\big|_{\t=\t_0}=0, \quad \nabla_\t^2 D_{KL}(p_{\t_0}\|p_{\t})\big|_{\t=\t_0}=I(\t_0). \]
We can consider a second order Taylor series expansion of $D_{KL}(p_{\t_0}\|p_{\t})$ in the variable $\t$ in a small enough neighbourhood of $\t_0$, and obtain 
\begin{equation} \label{eq:degenerate-1}
D_{KL}(p_{\t_0}\|p_{\t}) \ge c \sigma_{\min}(I(\t_0))\varrho(\t,\t_0)^2,
\end{equation}
where  $\sigma_{\min}(I(\t_0))$ is the smallest singular value of $I(\t_0)$. 

We combine the above observation with Lemma \ref{lem:lbound-impossibility} and Proposition \ref{prop:quadratic_UB}. For $\t_k$ as in Lemma \ref{lem:lbound-impossibility}, we may deduce via Proposition \ref{prop:quadratic_UB} that 
\begin{equation} \label{eq:degenerate-2}
D_{KL}(p_{\t_0}\|p_{\t}) \le c \|\t_0\|_2^4 \s^{-6} \varrho(\t,\t_0)^2  \le  c(L) \s^{-6} \varrho(\t,\t_0)^2,
\end{equation}
where we have used the fact that the boundedness of the signal class $\cT$ implies that $\|\t\|_2 \le C(L)$ uniformly for $\t \in \cT$ for some positive number $C(L)$.

Combining \eqref{eq:degenerate-1} and \eqref{eq:degenerate-2}, we obtain
\begin{equation} \label{eq:degenerate-3}
\sigma_{\min}(I(\t_0))  \le c(L) \s^{-6}.
\end{equation}

We recall from Section \ref{sec:est_rates_curv}, in particular \eqref{eq:asymp_normality}, that $\sqrt{n}(\tlt_n-\t_0) \to N(0,I(\t_0)^{-1})$ as $n \to \infty$. Let $\cZ$ denote a standard Gaussian vector of the same dimension as $\t_0$. Then we may write  $\sqrt{n}(\tlt_n-\t_0) \to I(\t_0)^{-1/2}\cZ$. 
Write $I(\t_0)=U^* \Sigma U$ as the spectral decomposition of $I(\t_0)$ with the eigenvalues $\{\s_i(I(\t_0))\}_i$. Notice that, by  rotational invariance,  $\cZ':=U\cZ$ is also a standard Gaussian of the same dimension (with co-ordinates $\{\cZ_i'\}_i$ being distributed as a standard $N(0,1)$ variable $z$).  We may then deduce that 
\begin{align*}
 & \sqrt{n}\|\t-\t_0\|_2 \to  \|  I(\t_0)^{-1/2}\cZ \|_2 \\ = &  \| U^* \Sigma^{-1/2} U \cZ \|_2  = \| \Sigma^{-1/2} \cZ' \|_2 \\  = &  \sqrt{\sum_i \s(I(\t_0))^{-1} |\cZ_i'|^2} \ge \sqrt{[\s_{\min}(I(\t_0))]^{-1}} |z|  \\ \ge & c_1(L) \s^3 |z| = \Omega_p(\s^3 ), 
 \end{align*} 
 where in the last inequality we have used \eqref{eq:degenerate-3}.
%

Finally, in a small enough neighbourhood of $\t_0$ we may identify $\vr(\t,\t_0)$ as $\frac{1}{\sqrt{L}}\|\t-\t_0\|_2$, and conclude that $\sqrt{n}\vr(\tlt_n,\t_0) = \Omega_p(\s^3)$ as $n \to \infty$, as desired. 
\end{proof}

We henceforth focus our attention on proving Lemma \ref{lem:lbound-impossibility}.

In order to carry out our investigations, we will utilise the Fourier transform as a key tool, and make repeated use of the renowned Parseval-Plancherel Theorem regarding the isometry properties of the Fourier transform. For convenience, we recall below these notions relevant for  our analysis on $\Z_L$.

\begin{defi}
For $\t = (\t_1,\ldots,\t_L) \in \R^L$, i.e. $\t: \Z_L \mapsto \R$, we define the Fourier transform $\hht: \Z_L \mapsto \C$ as
\[ \hht_j = \sum_{k=1}^L \t_k \exp\left( - \frac{2\pi i jk}{L} \right). \]
The inverse Fourier transform of $\t$, denoted $\ct: \Z_L \mapsto \C$ is defined as
\[ \ct_j = \frac{1}{L} \sum_{k=1}^L \t_k \exp\left(  \frac{2\pi i jk}{L} \right). \]
\end{defi}

\begin{thm}[Parseval-Plancherel Theorem]  \label{thm:Plancherel}
For $\t = (\t_1,\ldots,\t_L) \in \R^L$, i.e. $\t: \Z_L \mapsto \R$, we have
\[ \|\t\|_2^2 = \sum_{k=1}^L |\t_k|^2 = \frac{1}{L} \sum_{k=1}^L |\hht_k|^2 = \frac{1}{L}\|\hht\|_2^2. \] 
Equivalently, for $\t,\psi \in \R^L$, we may write
\[  \sum_{k=1}^L \t_k \overline{\psi}_k = \frac{1}{L} \sum_{k=1}^L \hht_k \overline{\hat{\psi}_k}  = \frac{1}{L} \cdot \lg \hht , \hat{\psi} \rg . \] 
\end{thm}

We begin with a result that obtains a succinct expression for the second moment tensor that is valid for any signal in $\R^L$, and therefore of general interest. For stating our result, we introduce the following notation: for any vector $v \in \R^L$, we denote by $\cM(v)$ the matrix \[[\cM(v)]_{ij}:=v(i-j).\] This identifies the matrix $\cM(v)$ as the Toeplitz matrix with symbol $\hat{v}$, the Fourier transform of $v$. Furthermore, as is common in our context, we will view  any vectors $u,v \in \R^L$ to be functions mapping $\Z_L \mapsto \R$, and by the convolution $u \ast v$ we will denote the convolution of these two functions (under the action of the rotation group $\Z_L$). Namely, \[[u  \ast v](k)=\sum_{g \in \Z_L} u(g)v(k-g).\] We recall that for any $v: \Z_L \mapsto \R$, the vector $\check{v}$ is given by $[\check{v}](i)=v(-i)$.
\begin{lem} \label{lem:Toeplitz}
For any $v_1,v_2 \in \R^L$, we have \[\E_{\bG}[G (v_1 \otimes v_2) G^*]=\frac{1}{L}\cdot \cM(v_1 \ast \check{v}_2). \]
\end{lem}
\begin{proof}
Observe that \[\E_\cG[G (v_1 \otimes v_2) G^*]_{ij}=\frac{1}{L}\cdot \left(\sum_{g \in \Z_L} v_1(i+g)v_2(j+g) \right),\] where as is usual in this context, the indices $i+g,j+g$ for the co-ordinates of the vectors are interpreted in the cyclic group $\Z_L$. Setting $i'=i+g$, we can re-write
 \begin{align*}&\E_\cG[G (v_1 \otimes v_2) G^*]_{ij} \\= &\frac{1}{L}\cdot \left(\sum_{i' \in \Z_L} v_1(i')v_2(j-i+i') \right) \\=&\frac{1}{L}\cdot \left(\sum_{i' \in \Z_L} v_1(i')\check{v}_2(i-j-i') \right) \\=&\frac{1}{L}\cdot \left[\cM(v_1 \ast \check{v}_2)\right]_{ij} ,\end{align*}
 as desired.
\end{proof}
Another result which would be useful for us subsequently is the following lemma.
\begin{lem} \label{lem:difference}
For any $v \in \R^L$, we have
\begin{equation} \label{eq:norm} \|\cM(v)\|_F = \sqrt{L}\|v\|_2 = \|\hat{v}\|_2^2.  \end{equation}
More generally, \begin{equation} \label{eq:iprod} \mathrm{Tr}[\cM(v)\cM(w)^*]= L \lg v,w \rg = \lg \hat{v},\hat{w}  \rg. \end{equation}
\end{lem}
\begin{proof}
We have,
\begin{align*}
&\|\cM(v)\|_F^2\\
=& \sum_{i=1}^L \sum_{j=1}^L |v(i-j)|^2 \\
=& \sum_{k \in \Z_L} L|v(k)|^2 \\
=& L\|v\|_2^2 \\
=& \|\hat{v}\|_2^2 \quad \text{[by the Parseval-Plancherel Theorem]}
\end{align*}
which completes the proof of \eqref{eq:norm}.

The equality \eqref{eq:iprod} follows from \eqref{eq:norm} via polarization, wherein we make use of the fact that $\|\cM(v)\|_F^2=\mathrm{Tr}[\cM(v)\cM(v)^*]$ and that the mapping $v \mapsto \cM(v)$ is linear.
\end{proof}

We are now ready to state the following lemma.
\begin{lem} \label{lem:2nd-moment}
For any $\t,\phi \in \R^L$ such that $h=\phi-\t$, we have \[ \E_{\bG}[(G\phi)^{\otimes 2}] - \E_{\bG}[(G\t)^{\otimes 2}] = \frac{1}{L} \cdot \left(\cM(\t \ast \check{h}) + \cM(\check{\t} \ast h) + \cM(h \ast \check{h})   \right). \]
\end{lem}
\begin{proof}
We have,
\begin{align*}
&\E_{\bG}[(G\phi)^{\otimes 2}] - \E_{\bG}[(G\t)^{\otimes 2}]  \\
= &\E_{\bG}[(G(\t + h))^{\otimes 2}] - \E_{\bG}[(G\t)^{\otimes 2}] \\
= &\E_{\bG}[(G\t)^{\otimes 2}]+ \E_{\bG}[(Gh)^{\otimes 2}] + \E_{\bG}[G(\t \otimes h)G^*] + \E_{\bG}[G(h \otimes \t)G^*] - \E_{\bG}[(G\t)^{\otimes 2}] \\
= &  \E_{\bG}[G(\t \otimes h)G^*] + \E_{\bG}[G(h \otimes \t)G^*] + \E_{\bG}[(G h)^{\otimes 2}] \\
= & \frac{1}{L} \cdot \left[ \cM(\t \ast \check{h}) +  \cM(h \ast \check{\t}) + \cM( h \ast \check{h})   \right]   \\
 & \quad \text{(using Lemma \ref{lem:Toeplitz})} \\
\end{align*}
as desired.
\end{proof}
Lemma \ref{lem:2nd-moment} shows, in particular, that in the regime of small $h$, the linearisation (in $h$) of the second moment difference tensor is given by $J(\t,h):=\frac{1}{L} \cdot \left(\cM(\t \ast \check{h}) + \cM(\ct \ast h)  \right)$, which we will focus on in the proof of Lemma \ref{lem:lbound-impossibility} that follows.

\begin{proof}[Proof of Lemma \ref{lem:lbound-impossibility}]
In order to construct a sequence $\{\t_k\}_{k}$ as in the statement of the present Lemma, it would suffice to demonstrate the existence of  $\t$ arbitrarily close to $\t_0$ such that $\bt=\bt_0$ and \eqref{eq:lb-imp} is satisfied . We recall that, for $\vr(\t,\t_0)$ small enough, we may take $\vr(\t,\t_0)=\|\t-\t_0\|_2$.

In what follows, we will set $h:=\t-\t_0$.

We begin with
\begin{align*}
&\|\cM(\t_0 \ast \ch) + \cM(\ct_0 \ast h)\|_F^2 \\
= & \|\cM(\t_0 \ast \ch)\|_F^2 + \|\cM(\ct_0 \ast h)\|_F^2 + \mathrm{Tr}[\cM(\t_0 \ast \ch)\cM(\ct_0 \ast h)^*] \\
& \qquad \qquad + \mathrm{Tr}[\cM(\t_0 \ast \ch)^* \cM(\ct_0 \ast h)] \\
= &  \|\cM(\t_0 \ast \ch)\|_F^2 + \|\cM(\ct_0 \ast h)\|_F^2 + 2 \Re\left(\mathrm{Tr}[\cM(\t_0 \ast \ch)\cM(\ct_0 \ast h)^*]\right) \numberthis \label{eq:T0} \\
\end{align*}

Using Lemma \ref{lem:difference}, we deduce that $\|\cM(\t_0 \ast \ch)\|_F = \sqrt{L} \|\t_0 \ast \ch\|_2$. Using Parseval-Plancherel's Theorem, we deduce that \[\|\t_0 \ast \ch\|_2= \frac{1}{\sqrt{L}}\|\widehat{\t_0 \ast \ch}\|_2=\frac{1}{\sqrt{L}}\|\hht \cdot \hat{\ch}\|_2,\] where for two vectors $u,v \in \R^L$, the quantity $u  \cdot v$ denotes their co-ordinate wise product.  We further introduce the notations that, for any vector $v \in \R^L$, we denote by $|v|$ the vector given by $|v|(i)=|v(i)| ~\forall i \in \Z_L$, and by $v^2$ we denote the vector given by $v^2(i)=v(i)^2 ~\forall i \in \Z_L$.
 We may deduce from definition that, for any vector $v \in \R^L$, we have $\hat{\check{v}}=\ol{\hat{v}}$, which further leads to
\begin{equation} \label{eq:T1}
\|\cM(\t_0 \ast \ch)\|_F^2 = L \cdot \frac{1}{L}\|\hht \cdot \ol{\hh}\|_2^2 =  \lg |\hht|^2, |\hh|^2 \rg.
\end{equation}
 Similarly, we can deduce that
  \begin{equation} \label{eq:T2}
  \|\cM(\ct_0 \ast h)\|_F^2 = \|\hat{\ct} \cdot \hh\|_2^2 =  \lg |\hht|^2, |\hh|^2 \rg.
  \end{equation}
and
\begin{align*}
&\Re\left(\mathrm{Tr}[\cM(\t_0 \ast \ch)\cM(\ct_0 \ast h)^*] \right) \\
= & L \cdot \Re \lg \t_0 \ast \ch, \ct_0 \ast h \rg  \quad \text{[Using Lemma \ref{lem:difference}]} \\
= & L \cdot \frac{1}{L} \Re \lg \widehat{\t_0 \ast \ch}, \widehat{\ct_0 \ast h} \rg \quad \text{[Using the Parseval-Plancherel Theorem]}  \\
= & \Re \lg \hht_0 \cdot \hat{\ch}, \hat{\ct}_0 \cdot \hh \rg \\
= & \Re \lg \hht_0 \cdot \ol{\hh}, \ol{\hht}_0 \cdot \hh \rg. \numberthis \label{eq:T3} \\
\end{align*}

Combining \eqref{eq:T0}-\eqref{eq:T3}, we may deduce that
\begin{align*}
&\|\cM(\t_0 \ast \ch) + \cM(\ct_0 \ast h)\|_F^2 \\
= &  2 \lg |\hht|^2, |\hh|^2 \rg + 2 \Re \lg \hht_0 \cdot \ol{\hh}, \ol{\hht}_0 \cdot \hh \rg \\
= & \left[   \sum_{\xi \in \Z_L} 2 \left( |\hht(\xi)|^2|\hh(\xi)|^2 + \Re \left(\hht(\xi)^2 [\ol{\hh(\xi)}]^2 \right)    \right)    \right] \numberthis \label{eq:Fourier-1} \\
\end{align*}

For $\t_0 \in T$ as in the statement of the Theorem, we propose to choose $h$ such that 
\begin{equation} \label{eq:h_choice}
\left( |\hht(\xi)|^2|\hh(\xi)|^2 + \hht(\xi)^2 [\ol{\hh(\xi)}]^2    \right)=0 ~\forall \xi \in \Z_L.
\end{equation} 
This would be possible because of the following reasons; using the fact that the Fourier transform is a bijection, we will determine the choice of $h$ in the Fourier domain..

First, we set $\hh(0)=0$, which will come in handy later. To set the coordinates $\hh(\xi)$ for $\xi \ne 0$, we proceed as follows. Recall that the only restriction on vectors in the signal class $\cT$ is that they have full support and their co-ordinates assume real values between $m$ and $M$. This translates into the fact that the only restriction on the difference $h$ of two such signals in the interior of $\cT$ is that $h$ has real co-ordinates (as long as $\|h\|_2$ is small enough). This implies that the only restriction on the Fourier transform $\hh$ is that $\hh$ is symmetric (the essential reason for which is that the Fourier transform is surjective from $\R^L$ to symmetric vectors in $\mathbb{C}^L$). So, for any given $\t_0 \in T$, we choose $\hh(\xi)$ (for $\xi \ne 0$) such that $|\hh(\xi)|$ is small enough and 
\begin{equation} \label{eq:angle}
2 \arg(\hh(\xi))=\pi - 2 \arg(\hht_0(\xi)),
\end{equation} 
which ensures that $\left( |\hht(\xi)|^2|\hh(\xi)|^2 + \hht(\xi)^2 [\ol{\hh(\xi)}]^2    \right)=0$. The symmetry condition on $\hh$ can be satisfied because, $\t_0 \in T$ implies that $\hht_0$ is symmetric, which allows us to choose $\hh$ as in \eqref{eq:angle} so that the vector $\hh$ is indeed symmetric. 

The upshot of the \eqref{eq:h_choice} is that $\|\cM(\t_0 \ast \ch) + \cM(\ct_0 \ast h)\|_F=0$, which implies that $\cM(\t_0 \ast \ch) + \cM(\ct_0 \ast h)=0$. Thus, Lemma \ref{lem:2nd-moment} implies that \[\|\E_{\bG}[(G\phi)^{\otimes 2}] - \E_{\bG}[(G\t)^{\otimes 2}] \|_F = \|\cM(h \ast \ch)\|_F. \]

But then we have
\begin{align*}
&\|\cM(h \ast \ch)\|_F^2
=  L \|h \ast \ch\|_2^2 \quad \text{[Using Lemma \ref{lem:difference}]}  \\
= &  L \cdot \frac{1}{L}  \|\widehat{h \ast \ch}\|_2^2 
=    \|\hh \cdot \ol{\hh}\|_2^2 \quad \text{[By Parseval-Plancherel Theorem]}  \\
= &  \left(\sum_{\xi \in \Z_L}  |\hh(\xi)|^4\right ) \numberthis \label{eq:4th-moment} \\
\le &  \left(\sum_{\xi \in \Z_L}  |\hh(\xi)|^2 \right)^2 \quad \text{[By Cauchy-Schwarz]}  \\
= & \|\hh\|_2^4 = L^2 \|h\|_2^4. \quad \text{[By Parseval-Plancherel]}  \numberthis \label{eq:higher-order}
\end{align*}

The upshot of this is that $\|\E_{\bG}[(G\phi)^{\otimes 2}] - \E_{\bG}[(G\t)^{\otimes 2}] \|_F \le C_L \|h\|_2^2$, thereby verifying the condition \eqref{eq:lb-imp}. 

It remains to verify the condition $\bt=\bt_0$; equivalently, $\overline{h}=0$. Observe that $\hh(0)=\sum_{x \in \Z_L} h(x) = L \overline{h}$. However, we have already set $\hh(0)=0$, which therefore implies that $\overline{h}=0$. This entails that $\bt=\bt_0$, thereby completing the proof.

\end{proof}

\section{The regime of moderate sparsity and uncertainty principles} \label{sec:mod-sparsity}
In this section, we establish Theorems \ref{thm:sparse-symm-G} and Lemma \ref{lem:symm-tech}, in the process invoking Uniform Uncertainty Principles from the discrete Fourier analysis of sparse vectors. We will proceed as follows. First, we will establish the technical Lemma \ref{lem:symm-tech}. Next, we will verify that the generic signals considered in Theorem \ref{thm:sparse-symm-G} satisfy the conditions of Lemma \ref{lem:symm-tech}, thereby completing the proof of Theorem \ref{thm:sparse-symm-G}. Finally, we will demonstrate the genericity of  support for the symmetric Bernoulli Gaussian and $N_{[-s,s]}^{\mathrm{symm}}(0,\zeta^2I)$ distributions. 

\subsection{Proof of Theorem \ref{thm:sparse-symm-G} and Lemma \ref{lem:symm-tech}} \label{sec:mod-sparsity-thmpf}

\begin{proof}[Proof of Lemma \ref{lem:symm-tech}] 
Our strategy would involve demonstrating a lower bound on the norm of the second moment difference tensor $\Delta_2(\varphi,\t)=\E_\bG[(G \varphi)^{\otimes 2}] - \E_\bG[(G\t)^{\otimes 2}]$ that is linear in the distance $\rho(\varphi,\t)$. That is, $\| \Delta_2(\varphi,\t)\|_F \ge C(L) \cdot \rho(\varphi,\t)$. We will do so for $\vp \in \cT$ lying in a neighbourhood of $\t$. Once such a lower bound is obtained, the theorem will follow from Proposition \ref{prop:linear_LB} and Proposition \ref{prop:CurveLB-normality} part (i).

We will work in the local neighbourhood of $\t$,  so that $\varrho(\varphi,\t)\le r(L)$; the precise size of $r(L)$ will be specified later. If $r(L)$ is small enough, then without loss of generality, we may write $\rho(\vp,\t)=\|\vp-\t\|_2=\|h\|_2$, which is the formulation that we will work with. 

Notice that, because of the lower bound in part (ii) of the statement of this Lemma, we have both $\min_\{i \in \supp(\vp) \} |\vp(i)| , \min_\{i \in \supp(\t) \} |\t(i)| \ge m_\cT$. As such, if $r(L)$ is small enough such that $r(L) < m_\cT$, we have $\|\vp-\t\|_2=\|h\|_2 < m_\cT$. This implies that, for $r(L)$ small enough, the supports of $\vp$ and $\t$ must coincide. In particular, the difference $h=\vp-\t$ must satisfy $\supp(h) \subseteq \supp(\t)$. This enables us to invoke condition (iii-a) in the statement of the Lemma for such $h$, which we shall use below.

Since $\t,\varphi$ are symmetric, $h=\varphi-\t$ are also symmetric. This implies that $\t=\ct$ and $h=\ch$, and both  Fourier transforms $\hht$ and $\hh$ are real-valued.
From Lemma \ref{lem:2nd-moment} we may deduce that
\begin{equation} \label{eq:2nd-moment-LB-1}
\|\E_{\bG}[(G\varphi)^{\otimes 2}] - \E_{\bG}[(G\t)^{\otimes 2}] \|_F \ge \frac{2}{L} \cdot \|\cM(\t \ast h)\|_F - \frac{1}{L}\|\cM(h \ast h)\|_F.
\end{equation}

Using \eqref{eq:Fourier-1} and \eqref{eq:4th-moment}, we may further simplify this to
\begin{align*}
\|\Delta(\vp,\t)\|_F = &\|\E_{\bG}[(G\varphi)^{\otimes 2}] - \E_{\bG}[(G\t)^{\otimes 2}] \|_F  \\
\ge &\frac{1}{L}\left(4\cdot  \sum_{\xi \in \Z_L} |\hht(\xi)|^2|\hh(\xi)|^2\right)^{1/2} - \frac{1}{L} \left( \sum_{\xi \in \Z_L}  |\hh(\xi)|^4\right)^{1/2}. \numberthis \label{eq:2nd-moment-LB-2}
\end{align*}
The second term on the right hand side will be controlled using the fact that $\|h\|_2 \le r(L)$, a consequence of the fact that $\vr(\vp,\t) \le r(L)$. To demonstrate this, we consider $\|h\|_{\infty}=\sup_{\xi \in \Z_L}|\hh(\xi)|$. For any $\xi \in \Z_L$, we observe via the Cauchy Schwarz inequality that 
\begin{equation} \label{eq:FT_upper_bound}
|\hh(\xi)|= \left|\sum_{k\in \Z_L}h(k)\exp\left(\frac{2\pi i k}{L}\right)\right| \le |\supp(h)| \cdot \|h\|_2 \le L \cdot \|h\|_2.
\end{equation}

Using the Parseval-Plancherel Theorem, we may proceed as
\begin{align*}
&\left(\sum_{\xi \in \Z_L} |\hh(\xi)|^4 \right)^{1/2} \\
\le & ~\|h\|_{\infty} \left(\sum_{\xi \in \Z_L}  |\hh(\xi)|^2 \right)^{1/2}  \\
\le & ~L  ~\|h\|_2\|\hh\|_2  \quad \text{[using \eqref{eq:FT_upper_bound}]} \\ 
\le & ~L \sqrt{L} ~\|h\|_2^2 \quad \text{[via Parseval-Plancherel Thoerem]}\\
\le & ~L \sqrt{L} ~r(L) \cdot \|h\|_2. \quad \text{[Since $\|h\|_2 \le r(L)$]} \numberthis \label{eq:symm-tech-UB}
\end{align*}
Thus, if we are able to show in the context of \eqref{eq:2nd-moment-LB-2} that $\left(4\sum_{\xi \in \Z_L} |\hht(\xi)|^2|\hh(\xi)|^2\right)^{1/2}$ is bounded below by $c(L)\|h\|_2$, then as soon as $r(L)$ is chosen to be small enough such that 
$ L  \sqrt{L} ~r(L)\le \frac{1}{2}c(L)$, we will be done (via \eqref{eq:2nd-moment-LB-2}) with an overall lower bound on $\|\Delta(\vp,\t)\|_F$ by $\frac{1}{2L} \cdot c(L) \cdot \|h\|_2$. 

In view of this, we will henceforth focus attention to lower-bounding  \linebreak $\left(4 \cdot  \sum_{\xi \in \Z_L} |\hht(\xi)|^2|\hh(\xi)|^2\right)^{1/2}$.

To this end, we recall the set $\L$ and the quantity $\fm(\cT)$ from the defining criteria of $\cT$, and proceed as
\begin{align*}
&\frac{1}{L} \cdot  \sum_{\xi \in \Z_L} |\hht(\xi)|^2|\hh(\xi)|^2 \quad = \frac{1}{L} \cdot  \sum_{\xi \in \Z_L} |\widehat{\t \ast h}(\xi)|^2 \\
\ge &c_2^{-1} \frac{1}{|\L|} \cdot  \sum_{\xi \in \L} |\widehat{\t \ast h}(\xi)|^2 \quad \text{(since $\t \ast h$ is 4s-sparse)} \\
= &c_2^{-1} \cdot \frac{1}{|\L|} \cdot  \sum_{\xi \in \L} |\hht(\xi)|^2|\hh(\xi)|^2 \quad \ge c_2^{-1} \cdot \frac{1}{|\L|} \cdot  \sum_{\xi \in \L}\fm(\cT)^2|\hh(\xi)|^2 \\
= &c_2^{-1} \fm(\cT)^2 \cdot \frac{1}{|\L|} \cdot  \sum_{\xi \in \L}|\hh(\xi)|^2 \quad \ge  c_1c_2^{-1} \fm(\cT)^2 \cdot \frac{1}{L} \cdot  \sum_{\xi \in \Z_L}|\hh(\xi)|^2 \\
= & c_3^2 \cdot \fm(\cT)^2\|h\|_2^2, \numberthis \label{eq:symm-tech-LB}
\end{align*}
where $c_3 = \sqrt{c_1c_2^{-1}}$.

We consider \eqref{eq:symm-tech-LB} in the context of \eqref{eq:symm-tech-UB} and the discussion immediately following it. With $c(L) = 2 c_3 \cdot \sqrt{L} \cdot \fm(\cT)$, we may conclude that $\left(4\sum_{\xi \in \Z_L} |\hht(\xi)|^2|\hh(\xi)|^2\right)^{1/2} \ge c(L)\|h\|_2$, and therefore $\|\Delta(\vp,\t)\|_F \ge \frac{1}{2L} \cdot c(L) \cdot \|h\|_2 \ge c_4 \cdot  \frac{\fm(\cT)}{\sqrt{L}} \cdot \|h\|_2$ for a suitable positive constant $c_4$. Recalling that $ \rho(\phi,\t)=\|h\|_2$, 
we obtain the desired lower bound on the second moment tensor difference in a neighbourhood of $\t$:
\begin{equation} \label{eq:curv_LB_moderate}
\|\Delta(\phi,\t)\|_F=\|\E_{\bG}[(G\varphi)^{\otimes 2}] - \E_{\bG}[(G\t)^{\otimes 2}]\|_F \ge c_4 \cdot \frac{\fm(\cT)}{\sqrt{L}} \cdot \rho(\phi,\t).
\end{equation}
Combined with the discussion at the beginning of this proof, this completes the argument.
\end{proof}

We now discuss the proof of Theorem \ref{thm:sparse-symm-G}.

\begin{proof}[Proof of Theorem \ref{thm:sparse-symm-G}]
Our approach to this proof will involve demonstrating that the set of signals $\cT=\cT_s$, as in the statement of Lemma \ref{lem:symm-tech}, has high probability under the generative model for the signal in the present theorem. We will do this by showing below that each of the criteria (i)-(iii) in Lemma \ref{lem:symm-tech} has high probability under the conditions of our current theorem.

\vspace{10pt}

\underline{(i)}: This condition is trivially satisfied by the present generative model, by definition.

\vspace{10pt}

\underline{(ii)}: We now consider the upper and lower bounds on the signal $\t$ on $\supp(\t)=\Xi$. In doing so, we will use the fact that the support $\Xi$ is \textit{typically $s$-sparse} with sparsity constants $(\a,\b)$, which implies that $\a s \le |\Xi| \le \b s$ with high probability.

But $\max_{k \in  \supp(\t)}|\t(k)| \le \max\{\xi_k:k\in \supp(\t)\}$, where the $\xi_k$-s are i.i.d. centred Gaussians with variance  $\vs^2$, is given by $O_p(\vs \sqrt{2\log |\Xi| }) =O_p( \vs \sqrt{\log s})$. This enables us to set $M_{\cT_s}=c \vs \log s$ in order to ensure that the maximum condition is satisfied with high probability.

Similarly, $\min_{k \in  \supp(\t)}|\t(k)| \le \min\{\xi_k:k\in \supp(\t)\}$ will  decay (in $|\Xi|$) as $x$ such that $x|\Xi|/\vs=O_p(1)$; this follows from the functional form of the Gaussian density $N(0,\vs^2)$. Thus, it suffices to take $m_{\cT_s}=c\min_{\Xi} \left( \vs/|\Xi|\log |\Xi| \right)=c\vs/s\log s$ in order to ensure that the minimum condition is satisfied with high probability.

\vspace{10pt}

\underline{(iii)}: We now come to the consideration of the desirable set of frequencies $\L$. We will demonstrate the existence of such a frequency set by an \textit{enhanced} version of the probabilistic method. While we will use randomness for finding the desirable set $\L$, it may be observed that the typical random frequency set of the right size would not satisfy the stipulated conditions on $\L$. Instead, we will require additional considerations in order to show the \textit{existence} of \textit{one such (atypical) set}. To this end, we would require certain auxiliary technical results, which are encapsulated in Lemmas \ref{lem:RV}, \ref{lem:good-set} and \ref{lem:neg-moment}.

To begin with, we
\[\text{Draw a random subset $\L \subset \Z_L$ having expected size $a$, that is } \] 
\begin{equation} \label{eq:random_set}
 \text{ each element of $\Z_L$ can be in $\L$ independently with probability $a/L$, } 
\end{equation}
where $a$ is a positive number slightly larger than the sparsity parameter $s$, to be specified in detail later.
By Lemma \ref{lem:RV}, such a set $\L$ will satisfy the condition (iii)(a) in the present Lemma for all $s$-sparse vectors $h$ with probability 
\begin{equation} \label{eq:s-ubound}
\P[\L \text{ satisfies (iii)(a) for all } s-\text{sparse vectors } h] \ge 1-5\exp(-c a(s \log^4 L)^{-1}), 
\end{equation}
the probability in question being in the randomness of $\L$. We say that a subset $A \subset \Z_L$  satisfies the Uniform Uncertainty Principle for $s$-sparse vectors (abbrv. s-UUP), if it satisfies \eqref{eq:RV}. Observe that $a$, being equal to $|\L|$ and $\L \subseteq \Z_L$, needs to satisfy $a \le L$. In view of this, to ensure that $\L$ satisfies (iii)(a) with high probability (as $s,L$ grow large), we need to have $a(s \log^4 L)^{-1} \to \infty$ in \eqref{eq:s-ubound}, which is equivalent to $s \ll L/(\log^4 L)$, which in turn is ensured by the condition $L \le  L/(\log^5 L)$ for $L$ large enough.

We now work towards showing that with positive, albeit vanishingly small probability, the condition (iii)(b) is also satisfied. To that end, we notice that while (iii)(a) is valid for all $s$-sparse vectors $h$,  we need to verify (iii)(b) only for the signal class of our interest - fact that we will crucially exploit in our considerations.

We begin with the observation that, since $\Xi$ is ($s^\tau$)-cosine generic, with high probability in the set $\Xi$ we have the inequality $\min_{\xi \in \Z_L} |\cC(\Xi,\xi)| \ge s^\tau$. For such a set $\Xi$, and any $\eta \in (0,1)$, we invoke the quantity $\fa(\Xi,\eta,\vs)=C(1-\eta)^{-1}\vs^{-\eta}\left(\min_{\xi \in \Z_L}\cC(\Xi,\xi)\right)^{-\frac{1}{2}\eta}$ (c.f. \eqref{eq:frak_a}), which immediately leads to  the bound $\fa(\Xi,\eta,\vs)\le C(1-\eta)^{-1}\vs^{-\eta}s^{-\frac{1}{2}\tau \eta}$.

For  $\k>0$ to be specified later, notice that this implies, with high probability in the signal $\t$, that 
\[\fa(\Xi,\eta,\vs) |\Xi|^{- \frac{1}{2} \k \eta} \le C_\eta \vs^{-\eta}s^{-\frac{1}{2}(\k+\tau)\eta}\] 
with $C_\eta=C(1-\eta)^{-1}$, which is small in the regime of large $s$ (and fixed $\vs$) as soon as $\k + \tau >0$. In light of Lemma \ref{lem:good-set} and the defining equation \eqref{eq:good-set}, we may therefore conclude that, with high probability in the signal $\t$, we have the inequality $|\hht(\xi)| \ge s^{-\k}$ for a set $\fS(\k)$ of frequencies $\xi$ satisfying $|\fS(\k)| \ge L(1-C_\eta \vs^{-\eta} s^{-\frac{1}{2}(\k+\tau)\eta})$.


 

For a given $\tau$, we now select $\k=\max\{4-\tau, 0\}$, which implies in particular that $\frac{1}{2}(\k+\tau) \ge 2$ and automatically $\k+\tau>0$. We then choose $\eta=3/4$, leading to the bound $|\fS(\k)| \ge L(1- c \vs^{-1/2} s^{-3/2})$. On this set $\fS(\k)$, with high probability in the signal $\t$, we have $|\hht(\xi)|\ge c \min\{s^{\tau-4},1\}$.

We now examine carefully a randomly sampled  subset $\L \subset \Z_L$ with average size $a$, as in \eqref{eq:random_set}. We want to understand $\P[\L \subset \fS(\k)]$, equivalently, $\P[\fS(\k)^\c \subset \L^\c]$.  Observe from the discussion above that $|\fS(\k)^\c|\le cL\vs^{-1/2}s^{-3/2}$, and note that the probability of a particular frequency $\xi$ to belong to $\L^\c$ is $(1-a/L)$. Since each frequency in $\Z_L$ is chosen to belong to $\L$ independently of each other,  we may deduce that, as long as $a/L$ remains bounded away from 1, we have
\begin{align*}
\P[\L \subset \fS(\k)] = &\P[\fS(\k)^\c \subset \L^\c] \\ =& \left( 1-a/L\right)^{|\fS(\k)^\c|} \\ \ge &\exp(-c'\frac{a}{L}\cdot cL\vs^{-1/2}s^{-3/2}) \\ = &\exp(-c''\vs^{-1/2}a s^{-3/2}).  \numberthis \label{eq:LB}
\end{align*}

We may then proceed as
\begin{align*}
&\P[ \{\L \text{ is s-UUP } \} \cap \{ \L \subset \fS(\k) \} ] \\
= &\P[  \{ \L \subset \fS(\k) \}  \setminus  \{\L \text{ is s-UUP } \} ^\c  ] \\
\ge &\P[   \L \subset \fS(\k)  ] -   \P[ \{\L \text{ is s-UUP } \} ^\c  ] \\
\ge & \exp(-c''\vs^{-1/2}a s^{-3/2}) - 5\exp(-c_3a s^{-1} \log^{-4} L)) \text{[c.f. \eqref{eq:LB} and Lemma \ref{lem:RV}]} \numberthis \label{eq:difference}
\end{align*}
The last expression in \eqref{eq:difference} is positive as soon as $\vs^{-1/2}s^{-3/2} \ll s^{-1} \log^{-4} L$, equivalently $s \gg \log^8 L/\vs$ in the regime of large $L$. The latter condition, in turn, is guaranteed by $s \ge \log^9 L$, as soon as $L$ is large enough.

In this regime, i.e. when $s \gg \log^8 L/\vs$, we have \[\P[ \{\L \text{ is s-UUP } \} \cap \{ \L \subset \fS(\k) \} ]>0,\] implying that there exists a realisation of the subset $\L$ such that:
\begin{itemize}
\item[(i)]  $\L$ satisfies \eqref{eq:RV} for all $s$-sparse signals.
\item[(ii)]  $ \min_{\xi \in \L} |\hht(\xi)| \ge c \min\{s^{\tau-4},1\}$.
\end{itemize}
These two facts together establish the existence of a frequency set $\L$ as required in condition (iii) of Lemma \ref{lem:symm-tech}, with
\begin{equation} \label{eq:explicit-LB}
\fm(\cT_s)=c \min\{s^{\tau-4},1\}
 \end{equation}
as above. This completes the argument for (iii), and therefore completes the proof of the present theorem.
\end{proof}

\begin{rem} \label{rem:L-free_LB}
We observe that in \eqref{eq:explicit-LB} the lower bound $\fm(\cT_s)=c \min\{s^{\tau-4},1\}$, in fact, depends only on $s$ (and not on $L$).
\end{rem}

\begin{rem} \label{lem:s-scaling}
In view of \eqref{eq:s-ubound} and the discussion immediately thereafter, we note that it suffices to have $s \le L/(\log^5 L)$ for $L$ large enough. On the other hand, in view of  \eqref{eq:difference} and the ensuing discussion, it suffices to have $s \ge \log^9 L$.  Combining these two observations, we work in the regime of $L$ such that 
\[ \log^9 L \le s \le L/(\log^5 L).\]
\end{rem}

It remains to establish Lemmas \ref{lem:good-set} -- \ref{lem:var}, which we take up in the next section. 

\subsection{Proofs of Lemmas \ref{lem:good-set} , \ref{lem:neg-moment} and \ref{lem:var}} \label{sec:mod-sparsity-lempf}
To begin with, we invoke the following \textit{Uniform Uncertainty Principle} from \cite{RudVer08}.
\begin{lem}[\cite{RudVer08}] \label{lem:RV}
Let $\fF$ be a random set of frequencies in $\Z_L$ having expected size $a$, that is each element of $\Z_L$ can be in $\fF$ independently with probability $a/L$. Then there are fixed numbers $c_1,c_2,c_3$ such that,  simultaneously for all  $f:\Z_L \mapsto \R$ that is $s-sparse$, the event
\begin{equation} \label{eq:RV}
c_1 \cdot \frac{1}{L} \sum_{ \xi \in \Z_L} |\hat{f}(\xi)|^2 \le \frac{1}{|\fF|}\sum_{ \xi \in \fF} |\hat{f}(\xi)|^2 \le c_2 \cdot \frac{1}{L} \sum_{ \xi \in \Z_L} |\hat{f}(\xi)|^2
\end{equation}
holds with probability $\ge 1-5\exp(-c_3a(s \log^4 L)^{-1})$.
\end{lem}

Next, for $\k>0$ and a function $f: \Z_L \to \R$ with $\supp(f)=\Xi$, we define the set of frequencies $\fS_f(\k)$ as
\begin{equation} \label{eq:good-set}
\fS_f(\k)=\{\xi \in \Z_L: |\hf(\xi)| \ge |\Xi|^{-\k} \}.
\end{equation}

Then we are ready to state the following lemma.
\begin{lem} \label{lem:good-set}
Let $\Xi \subset \Z_L$, let 
\begin{equation} \label{eq:frak_a}
\fa(\Xi,\eta,\vs)=C(1-\eta)^{-1}\vs^{-\eta}\left(\min_{\xi \in \Z_L}\cC(\Xi,\xi)\right)^{-\frac{1}{2}\eta} 
\end{equation} 
for $C>0$ as in Lemma \ref{lem:neg-moment} and any $0<\eta<1$,  and let $f \sim N_\Xi(\mathbf{0},\vs^2\mathbf{I})$. Then, for $\k>0$, we have \[|\fS_f(\k)| \ge L\bigg(1- \fa(\Xi,\eta,\vs)|\Xi|^{-\frac{1}{2}\k \eta}\bigg)\] with probability  $\ge 1-\fa(\Xi,\eta,\vs)|\Xi|^{-\frac{1}{2}\k \eta}$.
\end{lem}

\begin{proof}
We will approach this result by upper bounding the size of the set $\fS(\k)^\c$.
Observe that, $\xi \in \fS(\k)^\c$ implies that $|\hat{f}(\xi)|\le |\Xi|^{-\k}$; equivalently,
$|\hat{f}(\xi)|^{-\eta}\ge |\Xi|^{\k \eta}$.

This implies that,
\begin{align*}
|\Xi|^{\k \eta} |\fS(\k)^\c| = &\sum_{\xi \in \fS(\k)^\c} |\Xi|^{\k \eta} \\
\le &\sum_{\xi \in \fS(\k)^\c} |\hf(\xi)|^{-\eta} \\
\le &\sum_{\xi \in \Z_L}  |\hf(\xi)|^{-\eta}. \\
\end{align*}
Therefore, we may proceed as
\begin{equation}
|\Xi|^{\k \eta} \E[|\fS(\k)^\c|]
\le \E[\sum_{\xi \in \Z_L}  |\hf(\xi)|^{-\eta}]
\le \sum_{\xi \in \Z_L} \E[ |\hf(\xi)|^{-\eta}]\le L\fa(\Xi,\eta,\vs),
\end{equation}
where, in the last step, we make use of the definition \eqref{eq:frak_a} and of Lemma \ref{lem:neg-moment}; in particular choosing $C$ to be as in that Lemma.

We may restate this as
\begin{equation} \label{eq:exp-bound-goodset}
\E[|\fS(\k)^\c|] \le  L \fa(\Xi,\eta,\vs)|\Xi|^{-\k \eta}.
\end{equation}

We may then proceed via Markov's inequality as
\begin{align*}
\P[|\fS(\k)|\le & L(1-|\Xi|^{- \frac{1}{2} \k \eta})]  \\ = &\P[|\fS(\k)^\c|\ge L |\Xi|^{- \frac{1}{2} \k \eta}] \\ \le &\E[|\fS(\k)^\c|]/L|\Xi|^{- \frac{1}{2} \k \eta} \\
 \le &L^{-1} |\Xi|^{ \frac{1}{2} \k \eta} \cdot L \fa(\Xi,\eta,\vs) |\Xi|^{-\k \eta} \quad \text{[using \eqref{eq:exp-bound-goodset}]} \\
 = &\fa(\Xi,\eta,\vs)|\Xi|^{-\frac{1}{2}\k \eta},
\end{align*}
as desired.
\end{proof}


For $\Xi \subset \Z_L$ and $a \in \Z_L$, we recall \eqref{eq:def-cosine}:
\[
\cC(\Xi,a)=\mathbbm{1}_{\{0 \in \Xi\}} + 2 \sum_{k \in \Xi \setminus \{0\}} \cos^2(2\pi a k/L),
\]
where $\mathbbm{1}_A$
denotes the indicator function of the event $A$. 

\begin{lem} \label{lem:neg-moment}
Consider two subsets $\Xi, A \subset \Z_L$. Let $f \sim N_\Xi(\mathbf{0},\vs^2\mathbf{I})$. Then there is a positive number $C$ such that for any $0<\eta<1$, we have \[\E[|\hat{f}(\xi)|^{-\eta}] \le C(1-\eta)^{-1}\vs^{-\eta}\left(\min_{\xi \in A}\cC(\Xi,\xi)\right)^{-\frac{1}{2}\eta}\] for all $\xi \in A$.
\end{lem}

\begin{proof}
We invoke Lemma \ref{lem:var} in order to conclude that, for any $\xi \in \Z_L$, we have $\hf(\xi) \sim N(0,\vs^2\cC(\Xi,\xi))$. In other words, $\hf(\xi)$ is a 1D Gaussian random variable with variance $\cC(\Xi,\xi)$. Now, it follows from the  1D standard Gaussian density formula that for any $\eta<1$, a 1D standard Gaussian $Z$ has a finite negative moment $\E[|X|^{-\eta}]$  by $C (1-\eta)^{-1}$ for some positive number $C$. It follows that, if $Y \sim N(0,a^2)$, then $\E[|Y|^{-\eta}]=a^{-\eta}\E[|X|^{-\eta}] \le C (1-\eta)^{-1}a^{-\eta}$. This completes the proof.
\end{proof}

We can then state
\begin{lem} \label{lem:var}
Let $f \sim N_{\Xi}(\mathbf{0},\vs^2\mathbf{I})$. Then, for any $\xi \in \Z_L$, we have $\hf(\xi) \sim N(0,\vs^2\cC(\Xi,\xi))$.
\end{lem}

\begin{proof}
For any $\xi \in \Z_L$, we can write \[\hf(\xi)=\sum_{\k \in \Xi } f(k)\exp(2\pi i \xi k/L);\] using the fact that $\Xi$ is symmetric about the origin and denoting $\Xi_+=\Xi \cap \{1, \ldots,L/2\}$, this may be rewritten as
\begin{equation} \label{eq:fourier-exp}
\hf(\xi)=f(0) \mathbbm{1}_{\{0 \in \Xi\}}  + 2\sum_{\k \in \Xi_+}   f(k) \cos(2\pi \xi k/L).
\end{equation}
Since $\{f(k)\}_{k \in \Xi}$ is a collection of i.i.d. $N(0,\vs^2)$ random variables (and $0$ for $k \notin \Xi$), we deduce that $\hf(\xi)$ is Gaussian with mean 0 and variance $\vs^2\mathbbm{1}_{\{0 \in \Xi\}} + 4 \sum_{k \in \Xi_+} \vs^2 \cos^2(2\pi \xi k/L)=\vs^2 \cC(\Xi,\xi)$, as desired.
\end{proof}

\subsection{Genericity of support for symmetric Bernoulli Gaussian and $N_{[-s,s]}^{\mathrm{symm}}(0,\zeta^2 I)$ distributions} \label{sec:mod-sparsity_gensupp}
In this section, we demonstrate that two major classes of distributions in the regime of moderate sparsity -- namely, the sparse symmetric Bernoulli-Gaussian distribution and the $N_{[-s,s]}^{\mathrm{symm}}(0,\zeta^2I)$ distribution -- exhibit cosine-genericity of support. 

To begin with, we recall the definitions of support sets of signals being typically $s$-sparse (Definition \ref{def:typical-sparse}) and $\fb$-cosine generic (Definition \ref{def:cosgen}).  We would also need to make use of Bernstein's inequality, which we state below.
 
\begin{lem}[Bernstein's Inequality]\cite{boucheron2013concentration} \label{lem:Hoeffding}
Let $X_1,\ldots,X_n$ be mean zero random variables, and let $|X_i| \le M$ for all $1\le i \le n$. Then, for any $t>0$, we have
\[ \P\left( \sum_{i=1}^n X_i \ge t  \right) \le \exp\left( - \frac{\frac{1}{2}t^2}{\sum_{i=1}^n \E[X_i^2] + \frac{1}{3}Mt} \right). \]
\end{lem}

 We first take up the support properties of the symmetric Bernoulli-Gaussian distribution.
\begin{lem} \label{lem:cosgen-symBer}
For  $L$ large enough and $\log^9 L \le s \le L/\log^5 L$,  the sparse symmetric Bernoulli distribution with  mean 0, variance $\zeta^2$ and sparsity parameter $s$ is typically $s$-sparse with sparsity constants $(1/2 , 2)$, and is $s/32$-cosine generic.
\end{lem}

\begin{proof}
We first show that the sparse symmetric Bernoulli distribution is typically $s$-sparse, which amounts to showing that $|\Xi|$ is of order $s$ with high probability. But we observe that $|\Xi|=Y_0+ 2\sum_{i=1}^{\lfloor (L-1)/2 \rfloor}Y_i$, where each $Y_i$ is a Bernoulli($s/L$) random variable.
We immediately conclude that $\E[|\Xi|]=s+1$. Applying Bernstein's inequality (c.f. Lemma \ref{lem:Hoeffding}) to the centred random variables $X_i=2(Y_i-\E[Y_i])$ (for $i\ge 1$) and $X_0=Y_0-\E[Y_0]$ with $t=\frac{1}{2}s$ and $M=2$, we obtain
\[ \P\left( \big| |\Xi|  - \E[|\Xi|] \big| \ge t \right) \le \exp\left( - \frac{\frac{1}{8}s^2}{ 4(\frac{L}{2}+1)\frac{s}{L}(1- \frac{s}{L}) + \frac{1}{3}s} \right)=\exp(-cs(1+o_L(1))) \]
for some positive number $c$. Since $\E[\Xi]=s+1$, we deduce that $\frac{1}{2}s \le |\Xi| \le 2s$ with probability $1-o_L(1)$, implying that the sparse symmetric Bernoulli distribution is typically $s$-sparse, with sparsity constants $(1/2,2)$.

To demonstrate that the  symmetric Bernoulli distribution is cosine generic with the parameters as claimed in the statement of this lemma, we first compute, for a fixed $\xi \in \Z_L$, the expectation $\E[\cC(\Xi,\xi)]$. To this end, we may write
\begin{equation} \label{eq:var-expression}
\cC(\Xi,\xi)= Y_0  + 4 \sum_{k=1}^{{\lfloor (L-1)/2 \rfloor}} \cos^2(2\pi \xi k/L)Y_i,
\end{equation}
where the random variables $Y_i$ are defined as above.
Then
\begin{equation} \label{eq:exp-var-primary}
\E[\cC(\Xi,\xi)]=\frac{s}{L}+\frac{4s}{L}\sum_{k=1}^{{\lfloor (L-1)/2 \rfloor}} \cos^2(2\pi \xi k/L).
\end{equation}
Setting $\mu=\exp(2\pi i \xi/L)$, this reduces to
\begin{align*}
\E[\cC(\Xi,\xi)]= &\frac{s}{L} +  \frac{s}{L}\sum_{k=1}^{{\lfloor (L-1)/2 \rfloor}} |\mu^k+\mu^{-k}|^2= \frac{s}{L} + \frac{s}{L}  \sum_{k=1}^{{\lfloor (L-1)/2 \rfloor}} (2+2\Re(\mu^{2k})) \\
                   =& s(1+o_L(1)) + s \cdot \frac{1}{L}\Re\left( \sum_{k=1}^{{\lfloor (L-1)/2 \rfloor}} \mu^{2k} \right)  \\ = &s \cdot \left[1 + \frac{2}{L}\Re\left( \mu^2 \cdot \frac{1-\mu^{L-\alpha}}{1-\mu^2}\right) + o_L(1) \right], \numberthis \label{eq:exp-var}
                         \end{align*}
where $\alpha=1$ or $2$, depending on whether $L$ is odd or even.

By considering the magnitude of the quantity $\left( \mu^2 \cdot \frac{1-\mu^{L-\alpha}}{1-\mu^2}\right)$, we deduce from \eqref{eq:exp-var} that for large $s,L$ the expectation
\begin{align*}
&\E[\cC(\Xi,\xi)] \ge s/2  \\
& \text{ unless } |1-\mu^2|<8/L \iff |2\xi/L - \delta| <8/L \text{ for } \delta=0,\pm 1 \\
& \qquad \qquad \qquad \qquad \qquad \iff |\xi - \delta \cdot \frac{L}{2}| < 4   \text{ for } \delta=0,\pm 1. \numberthis      \label{eq:badset}
\end{align*}

It remains to deal with the frequencies $\xi$ that satisfy \eqref{eq:badset}. We will demonstrate the details for the case $\delta=0$; the computations for $\delta=\pm 1$ are similar, and indeed can be reduced to the consideration of $\delta=0$ by making a change of variables $\hat{\xi}=\xi-\delta \cdot \frac{L}{2}$ and observing that $\cos^2(2\pi \hat{\xi} k/L)=\cos^2(2\pi \xi k/L)$.

Therefore, we reduce ourselves to considering the frequencies $\xi$ (in the context of  \eqref{eq:badset}) such that $|\xi| < 4$. We then invoke \eqref{eq:exp-var-primary} and lower bound
\begin{equation} \label{eq:exp-var-secondary}
\E[\cC(\Xi,\xi)] \ge \frac{4s}{L}\sum_{k=1}^{L/32} \cos^2(2\pi \xi k/L) \ge \frac{4s}{L}\sum_{k=1}^{L/32} \cos^2(2\pi /8) = \frac{4s}{L} \cdot \frac{L}{32} \cdot \frac{1}{2} = s/16.
\end{equation}

We combine our analyses of the two classes of frequencies, summarize it as:
\begin{equation} \label{eq:exp-var-final}
\E[\cC(\Xi,\xi)] \ge s/16 \quad \forall \xi \in \Z_L.
\end{equation}


We centre the $Y_i$-s in \eqref{eq:var-expression} by their expectations, and define the centered random variables $X_0=Y_0-\E[Y_0]$ and for $1\le i \le {\lfloor (L-1)/2 \rfloor}$, $X_i=4\cos^2(2\pi \xi k/L)(Y_i-\E[Y_i])$. Then we may write \[\cC(\Xi,\xi)-\E[\cC(\Xi,\xi)]= \sum_{i=0}^{{\lfloor (L-1)/2 \rfloor}}X_i. \]

Notice that, for any $i \ge 0$, $\var[X_i]\le 4s/L(1-s/L)$, so that \[\sum_{i=0}^{{\lfloor (L-1)/2 \rfloor}} \E[X_i^2] \le 2s(1-s/L)(1+o_L(1)).\]

Applying Bernstein's inequality (c.f. Lemma \ref{lem:Hoeffding}) with $t=s/32$ and $M=4$, we proceed as
 \begin{align*}
 &\P\left( \big| \cC(\Xi,\xi)  - \E[|\cC(\Xi,\xi)] \big| \ge s/32 \right)  \\
& \le  \exp\left( - \frac{\frac{1}{2}t^2}{\sum_{i=0}^{{\lfloor (L-1)/2 \rfloor}} \E[X_i^2] + \frac{1}{3}Mt} \right) \\
 \le &\exp\left( - \frac{\frac{1}{2048}s^2}{2s(1-s/L)(1+o_L(1)) + \frac{1}{24}s} \right). \\
\le &\exp(-s/10^4(1+o_L(1))).
 \end{align*}
By a union bound, we may further deduce that
\begin{equation} \label{eq:union-bound}
\P\left( \exists \xi \in \Z_L \text{ such that }  \big| \cC(\Xi,\xi)  - \E[\cC(\Xi,\xi)] \big| \ge s/32 \right) \le L\exp(-s/10^4(1+o_L(1))).
\end{equation}
The right hand side is $o_L(1)$ as soon as $s \gg 10^4 \log L $.

On the complement  of the event in \eqref{eq:union-bound}, that is, when $\{\big| \cC(\Xi,\xi)  - \E[\cC(\Xi,\xi)] \big| < s/32 \quad \forall \xi \in \Z_L\}$, we may deduce from \eqref{eq:badset} that for all $\xi \in \Z_L$
\begin{align*}
&\cC(\Xi,\xi) \\
\ge & \E[\cC(\Xi,\xi)] - \big| \cC(\Xi,\xi)  - \E[\cC(\Xi,\xi)] \big| \\
\ge & s/16 - s/32 \\
= & s/32.
\end{align*}

This shows that $\min_{\xi \in \Z_L} \cC(\Xi,\xi) \ge s/32$ with probability $1-o_L(1)$ (in the random subset $\Xi$), thereby establishing the claim that for  $s\ge c\log L$ and $L$ large enough, the random subset $\Xi$ is $s/32$-cosine generic.
\end{proof}

Finally, we end this section with a study of the support properties of the $N_{[-s,s]}^{\mathrm{symm}}(0,\zeta^2I)$ distribution.
\begin{lem} \label{lem:cosgen-symdet}
For $s,L$ large enough, the deterministic subset $\{[-s,s] \cap \Z_L\}$ is typically $s$-sparse and is $s/16$-cosine generic.
\end{lem}

\begin{proof}
The (deterministic) subset $\Xi$ has size exactly $2s+1$, therefore $\Xi$ is trivially \textit{typically $s$-sparse}. It remains to show the cosine genericity of $\Xi$.

For any $\xi \in \Z_L$, we have
\begin{align*}
& \cC(\Xi,\xi) \\
= & 1 + 2 \sum_{k \in [-s,s] \setminus \{0\}} \cos^2(2\pi \xi k/L) \\
= & -1 + 2 \sum_{k=-s}^{s} \cos^2(2\pi \xi k/L) \numberthis \label{eq:cosgen-symdet-1} \\ 
= &  -3 +  \sum_{k=0}^{s} |\exp(2\pi i \xi k/L)+\exp(-2\pi i \xi k/L)|^2.
\end{align*}
Setting $\o=\exp(2\pi i \xi /L)$, we may proceed as
\begin{align*}
\cC(\Xi,\xi)= & -3 +  \sum_{k=0}^{s} |\o^k+\o^{-k}|^2=-3 +  \sum_{k=0}^{s} (2+2\Re(\o^{2k})) \\
                  =& 2s-1 + 2\Re\left( \sum_{k=0}^{s} \o^{2k} \right) = 2s-1 + 2\Re\left(  \frac{1-\o^{2s+2}}{1-\o^2}\right).
\end{align*}

The last equation implies that
\begin{equation} \label{eq:cosgen-symdet-2}
\cC(\Xi,\xi) \ge s(1+o(1))
\end{equation}
unless $\big|\Re\left( \frac{1-\o^{s+2}}{1-\o^2}\right)\big| > s/4$, which would in particular imply that
\begin{equation} \label{eq:badvar1}
\big|\frac{1-\o^{s+2}}{1-\o^2}\big| > s/4.
\end{equation}
Recalling the definition of $\o$, and observing that $|1-\o^{s+2}|\le 2$  we may deduce that \eqref{eq:badvar1} is true only if $|1-\o^2|<8/s$. Recalling that $\o=\exp(2\pi i \xi /L)$ , we deduce that   for large enough $s$, the inequality holds $|1-\o^2|<8/s$  only if $|2\xi/L - \delta \cdot \frac{L}{2}|\le 8/s (1+o_s(1))$, where $\delta=0, \pm 1$. As in the proof of Lemma \ref{lem:cosgen-symBer}, we focus on the case $\delta=0$, noting in passing that the cases $\delta = \pm 1$ are similar and are easily dealt with using a simple change of variables from $\xi$.

When $\delta=0$, we are considering frequencies $\xi$ such that $\xi/L < 4/s$.  This in particular implies that for all $|k| \le s/32$, we have $|2\pi \xi k/L| <\pi/4$, implying $\cos^2(2\pi \xi k/L) \ge 1/2$.

We now proceed to lower bound $\cC(\Xi,\xi)$ for $\xi \in \Z_L$. If $\xi \in \Z_L$ is such that $\o=\exp(2\pi i \xi /L)$ \textit{does not} satisfy \eqref{eq:badvar1}, then by \eqref{eq:cosgen-symdet-2} we conclude that $\cC(\Xi,\xi) \ge s(1+o(1))$. 

If $\xi \in \Z_L$ is such that $\o=\exp(2\pi i \xi /L)$  satisfies \eqref{eq:badvar1}, then we proceed as follows. 
Using \eqref{eq:cosgen-symdet-1}, we may lower bound $\cC(\Xi,\xi)$ as
\begin{align*}
& \cC(\Xi,\xi)  \\
= & -1 + 2 \sum_{k=-s}^{s} \cos^2(2\pi \xi k/L) \\
\ge & 1 + 2 \sum_{1 \le |k| \le s/32}  \cos^2(2\pi \xi k/L) \\
\ge & s/16. \numberthis \label{eq:cosgen-symdet-3}
\end{align*}

Combining \eqref{eq:cosgen-symdet-2} and \eqref{eq:cosgen-symdet-3}, we deduce that $\cC(\Xi,\xi) \ge s/16 \quad \forall \xi \in \Z_L$, thereby showing that $\Xi$ is $s/16$-cosine generic and  completing the proof of the lemma.
\end{proof}


\section{Results  on the curvature of $D_{KL}$}  \label{sec:curv_prop}
In this section, we provide the proofs of several propositions pertaining to the curvature of the KL divergence for the MRA model.
\subsection{Moment difference tensors and $D_{KL}$}
\begin{proof}[Proof of Proposition \ref{prop:CurveLB-normality}]
We discuss (i); the case of (ii) would be similar. We recall that the probability distribution $p_\t$ as well as  $D_{KL}(p_\t \| p_\phi)$ are invariant under the action of $\cG$, i.e., invariant under the transformations $\t \mapsto G \cdot \t$ for $G \in \cG$. As a result, for $\varrho(\t,\t_0)$ small enough (equivalently, $\|\t-\t_0\|_2$ small enough), we may assume without loss of generality that $\varrho(\t,\t_0)= \frac{1}{\sqrt{L}}\|\t-\t_0\|_2$ (c.f., \cite{BanRigWee17}; esp. the proof of Theorem 4 therein). 

The dimension of the Hessian of $D_{KL}(p_{\t_0}||p_\t)$ depends on the local dimension of the parameter space at the point $\t_0$, which is the same as $k=|\supp(\t_0)|$. The lower bound on $D_{KL}$ in (i) implies that 
\[  K_1(\sigma)^{1/2}\mathrm{Id}_k  \preceq I(\theta_0)   \iff  I(\theta_0)^{-1} \preceq K_1(\sigma)^{-1/2}\mathrm{Id}_k,\] 
where $\mathrm{Id}_k$ is the $k \times k$ identity matrix, and $\preceq$
denotes domination in the sense of non-negative definite matrices. Since $\sqrt{n}(\tlt_n - \theta_0) \to N(0,I(\theta_0)^{-1})$, setting $\mathcal{Z}_k \sim N(0,\mathrm{Id}_k)$, we may deduce that as $n \to \infty$ we have the distributional convergence
\begin{equation} \label{eq:asymptotics-1}
\sqrt{n}\|\tlt_n-\theta_0\|_2 = \sqrt{n\|\tlt_n-\theta_0\|_2^2} \to \|I(\theta_0)^{-1/2}\mathcal{Z}_k\|_2 .
\end{equation}
On the other hand, we have
\begin{equation} \label{eq:asymptotics-2}
 \|I(\theta_0)^{-1/2}\mathcal{Z}_k\|_2 = \sqrt{\langle \mathcal{Z}_k, I(\theta_0)^{-1} \mathcal{Z}_k \rangle} \le K_1(\sigma)^{-1/2} \|\mathcal{Z}_k\|_2.
\end{equation}
Thus, $\sqrt{n}\vr{\t,\t_0} = \sqrt{n} \frac{1}{\sqrt{L}} \cdot \|\tlt_n-\theta_0\|_2  =O_p( K_1(\sigma)^{-1/2}/\sqrt{L})$, as desired. 

We note in passing that $\|\mathcal{Z}_k\|_2$ is a $\sqrt{\chi^2(k)}$ distribution. 
\end{proof}

Recall that  for any $\t  = (\t_1,\ldots,\t_L) \in \R^L$, we denote $\bt=\frac{1}{L}\sum_{i=1}^L \t_i$.
This leads us to the fact that $\t^*:=\E_\cG[G\t]=\bt \cdot \ind$, where $\ind = (1,1,\ldots,1) \in \R^L$ is the all ones vector in $L$ dimensions. Finally, we denote by $\tlt$ the centred version of $\t$, that is, $\tlt=\t-\t^*=\t-\E_\cG[G\t]$. We observe that 
\begin{equation} \label{eq:centre_mean_zero}
\E_\cG[G\tlt]= \E_\cG[\t-\t^*]=\E_\cG[\t]-\E_\cG[\t^*] = \t^* - \t^* =0. 
\end{equation}

Notice further that, with the above notations, we may write
\begin{equation} \label{eq:delta1}
\Del_1(\t,\phi)=(\bt - \bp) \ind.
\end{equation}

Towards the proofs of Propositions \ref{prop:linear_LB} and \ref{prop:quadratic_UB}, we will now present a comparison between the second moment difference tensors for the centred and uncentred versions of two vectors $\t$ and $\phi$. To this end, we state the following Proposition.
\begin{prop} \label{prop:centred_vs_uncentred}
We have,
\[ \Del_2(\t,\phi) =  \Del_2(\tlt,\tlp) + (\bt^2 -\bp^2) \cdot \ind \otimes \ind.\]
\end{prop}

\begin{proof}
We have,
\begin{align*}
& \E_\cG[(G \t)^{\otimes 2}] \\
= & \E_\cG[(G(\tlt + \t^*))^{\otimes 2}]\\
= & \E_\cG[(G\tlt + G\t^*))^{\otimes 2}]\\
= & \E_\cG[(G\tlt + \t^*))^{\otimes 2}]  \quad \text{[since $\t^*$ is $\cG$-invariant]} \\
= & \E_\cG[(G\tlt)^{\otimes 2}] +  \E_\cG[G\tlt \otimes \t^*]  + \E_\cG[\t^* \otimes G\tlt ]  + \t^* \otimes \t^* \\
= & \E_\cG[(G\tlt)^{\otimes 2}] +  \E_\cG[G\tlt] \otimes \t^*  + \t^* \otimes  \E_\cG[G\tlt ]  + \t^* \otimes \t^* \\
= & \E_\cG[(G\tlt)^{\otimes 2}] + \t^* \otimes \t^* \quad \text{[using \eqref{eq:centre_mean_zero}]} \\
= & \E_\cG[(G\tlt)^{\otimes 2}] + \bt^2 \ind \otimes \ind. \numberthis \label{eq:cent_vs_uncent_1}
\end{align*}

In view of \eqref{eq:cent_vs_uncent_1}, we may write
\begin{equation} \label{eq:cent_vs_uncent_2}
 \Del_2(\t,\phi) = \E_\cG[(G \t)^{\otimes 2}] - \E_\cG[(G \phi)^{\otimes 2}] =  \Del_2(\tlt,\tlp) + (\bt^2 -\bp^2) \cdot \ind \otimes \ind, 
\end{equation}
as desired.
\end{proof}

We now proceed to establish Proposition \ref{prop:linear_LB}.
\begin{proof}[Proof of Proposition \ref{prop:linear_LB}]
Observe that, $|\bt^2 -\bp^2|=|\bt + \bp| \cdot |\bt -\bp| \le (\|\t\|_\infty + \|\phi\|_\infty) \cdot |\bt -\bp|$.
This implies, in particular, that 
\begin{equation} \label{eq:cent_vs_uncent_3}
\|\Del_2(\tlt,\tlp)\|_F \ge \|\Del_2(\t,\phi)\|_F - (\|\t\|_\infty + \|\phi\|_\infty) \cdot |\bt -\bp|\cdot \|\ind \otimes \ind \|_F.
\end{equation}

Now, Theorem \ref{thm:BRW2} implies that 
\begin{align*}
& D_{KL}(\tlt|| \tlp) \\ 
\ge & C \cdot \|\Del_2(\tlt,\tlp)\|_F^2/\s^4 \\  
\ge & C \cdot  \left(\|\Del_2(\t,\phi)\|_F - (\|\t\|_\infty + \|\phi\|_\infty) \cdot |\bt -\bp|\cdot \|\ind \otimes \ind \|_F\right)^2/\s^4 \\
\ge & C \cdot  \left(\frac{3}{4}\|\Del_2(\t,\phi)\|_F^2 - 3 (\|\t\|_\infty + \|\phi\|_\infty)^2 \cdot |\bt -\bp|^2 \cdot \|\ind \otimes \ind \|_F^2 \right)/\s^4
\end{align*}
for a positive number $C$, where in the last step we use Proposition \ref{prop:real_numbers}.

Combining the above with Lemma \ref{lem:BRW2} we obtain
\begin{align*} 
& D_{KL}(\t || \phi) \\ 
\ge & \frac{1}{2}|\bt - \bp|^2\|\ind\|_2^2 \cdot \s^{-2}  + C  \left(\frac{3}{4}\|\Del_2(\t,\phi)\|_F^2 - 3 (\|\t\|_\infty + \|\phi\|_\infty)^2  |\bt -\bp|^2  \|\ind \otimes \ind \|_F^2 \right) \\
\ge &    \s^{-4} \frac{3C}{4}\|\Del_2(\t,\phi)\|_F^2 + |\bt - \bp|^2 \left(\frac{1}{2}\s^{-2} \|\ind\|_2^2 - 3C\s^{-4}(\|\t\|_\infty + \|\phi\|_\infty)^2  \|\ind \otimes \ind \|_F^2 \right). 
\numberthis \label{eq:cent_vs_uncent_4}
\end{align*}

We now make use of the fact that the signal class $\cT$ is bounded  (in the deterministic setting), and in the case of generative models, the random signal is bounded with high probability. 

We then consider the term 
\[\left(\frac{1}{2}\s^{-2} \|\ind\|_2^2 - 3C\s^{-4}(\|\t\|_\infty + \|\phi\|_\infty)^2   \|\ind \otimes \ind \|_F^2 \right)\]
on the right hand side of \eqref{eq:cent_vs_uncent_4}, and observe that when $\s$ is large enough -- that is, $\s \ge \s_0(L)$ for some threshold $\s_0(L)$, we have
\begin{equation} \label{eq:control_mean_term}
\left(\s^{-2} \|\ind\|_2^2 - 3\s^{-4}(\|\t\|_\infty + \|\phi\|_\infty)^2   \|\ind \otimes \ind \|_F^2 \right) \ge \frac{1}{4}\s^{-2} \|\ind\|_2^2.
\end{equation}

Combining \eqref{eq:cent_vs_uncent_4} and \eqref{eq:control_mean_term}, we obtain 
\begin{align*} 
& D_{KL}(\t || \phi) \\
\ge &  \s^{-4} \cdot \frac{3C}{4}\|\Del_2(\t,\phi)\|_F^2 + \s^{-2}  \cdot \frac{1}{4}|\bt - \bp|^2  \|\ind\|_2^2 \\
\ge &  \s^{-4} \cdot \frac{3C}{4} \|\Del_2(\t,\phi)\|_F^2.  \numberthis \label{eq:cent_vs_uncent_5}
\end{align*}

\end{proof}

\begin{rem}
We observe that equality can hold in \eqref{eq:cent_vs_uncent_5}, whenever $\bt=\bp$. This is indeed possible for specific directions of approach of $\phi$ to the signal $\t$ when $\t$ lies in the interior of the signal class. The standard signal classes considered in MRA, in this paper as well as otherwise, and also the generative models  considered in this paper, have their interiors account for their full Lebesgue measure, so nearly all signals $\t$ do in fact have such a bad direction of approach where equality in \eqref{eq:cent_vs_uncent_5} holds.
\end{rem}

We continue on to the proof of Proposition \ref{prop:quadratic_UB}.

\begin{proof}[Proof of Proposition \ref{prop:quadratic_UB}]
When, for some $\t,\phi$, we have $\|\Del_2(\t,\phi)\|_F \le c \rho(\t,\phi)^2$, then we may proceed to analyse the order of $D_{KL}(\t \| \phi)$ as follows. Combining Lemma \ref{lem:BRW2} and Theorem \ref{thm:BRW2} applied with $k=3$, and noting that $\Delta_1(\tlt \| \tlp)=0$, we may proceed as
\begin{align*}
& D_{KL}(\t || \phi) \\
= &  D_{KL}(p_{\tilde{\theta}} || p_{\tilde{\phi}}) + \frac{1}{2\s^2} \|\Delta_1(\t,\phi)\|^2 \\
\le & 2 \sum_{m=1}^{2} \frac{\|\Delta_m(\tlt , \tlp)\|^2}{\sigma^{2m}m!} + C \frac{\|\tlt\|_2^{4}\rho(\tlt,\tlp)^2}{\sigma^{6}} +  \frac{1}{2\s^2} \|\Delta_1(\t,\phi)\|^2 \\
= &  \frac{1}{2\s^2} \cdot |\bt-\bp|^2 \|\ind \|_2^2 + \frac{1}{\s^4} \cdot  \|\Del_2(\tlt , \tlp)\|_F^2 + C \frac{\|\tlt\|_2^{4}\rho(\tlt,\tlp)^2}{\sigma^{6}}  \quad \text{[using \eqref{eq:delta1}]} \\
= &  \frac{1}{2\s^2} \cdot |\bt-\bp|^2 \|\ind \|_2^2 \left(1 + \frac{C_1\|\t\|_2^4}{\s^4} \right) + \frac{1}{\s^4} \cdot  \|\Del_2(\tlt , \tlp)\|_F^2 + \frac{2C \|\t\|_2^{4}}{\sigma^{6}}\rho(\t,\phi)^2, \numberthis \label{eq:degenerate}
\end{align*}
where, in the last step, we have used Proposition \ref{prop:distance_compare}.

Therefore, if $\t,\phi$ are such that $\bt=\bp$ and $\|\Del_2(\t,\phi)\|_F \le c \rho(\t,\phi)^2$,  we may conclude from \eqref{eq:degenerate} that 
\[ D_{KL}(\t || \phi)  \le \frac{1}{\s^4} \cdot C \cdot \rho(\t,\phi)^4 +  2C \frac{\|\t\|_2^4}{\sigma^6} \cdot \rho(\t,\phi)^2.\]
For small enough $\rho(\t,\phi)$, the quadratic term involving $\rho(\t,\phi)^2$ dominates in the above,
and we have
\begin{equation} \label{eq:degenerate_upperbound}
 D_{KL}(\t || \phi)  \le 4C \frac{\|\t\|_2^4}{\sigma^6} \cdot \rho(\t,\phi)^2
\end{equation}
for some positive number $C$ and small enough $\rho(\t,\phi)$.
\end{proof}



We complete this section with the auxiliary Propositions \ref{prop:real_numbers} and \ref{prop:distance_compare}.

\begin{prop} \label{prop:real_numbers}
Let $a,b>0$. Then we have
\[ (a-b)^2 \ge \frac{3}{4}a^2 - 3b^2. \]
\end{prop}

\begin{proof}
 We observe that
\begin{equation}  \label{eq:balance}
2ab = 2 \cdot \frac{1}{2}a \cdot 2b \le \frac{1}{4}a^2 + 4b^2. 
\end{equation}
We may then expand $(a-b)^2 = a^2+b^2-2ab$ and use \eqref{eq:balance} to lower bound the $-2ab$ term. This completes the proof.
\end{proof}

\begin{prop} \label{prop:distance_compare}
We have, 
\[ \rho(\tlt,\tlp)^2 \le 2 \rho(\t,\phi)^2 + 2|\bt - \bp|^2 \|\ind\|_2^2.  \]
\end{prop}

\begin{proof}
Recall that $\t^*=\E_\cG[G\t]$ is $\cG$-invariant, i.e., $G \t^*=\t^* \forall G \in \cG$; the same holds true for $\phi^*$.
For any $G \in \cG$, we may write
\begin{align*}
& \|\tlt - G \tlp\|_2^2 \\
= &  \|(\t-\t^*) - G (\phi - \phi^*)  \|_2^2 \\
= &  \|(\t - G \phi)  - (\t^*-\phi^*) \|_2^2 \quad \text{[Since $\t^*, \phi^*$ are $\cG$-invariant]} \\
\le & 2\|\t - G \phi \|_2^2 + 2 \| \t^*-\phi^* \|_2^2. \\
\end{align*}
We summarize the above computations as
\begin{equation} 
\|\tlt - G \tlp\|_2^2 \le 2\|\t - G \phi\|_2^2 + 2 \| \t^*-\phi^* \|_2^2 \quad \forall G \in \cG.
\end{equation}
This implies that, for any $G \in \cG$ we have
\begin{equation} \label{eq:summarize_dist_compare}
\rho(\tlt,\tlp)^2= \min_{\mathfrak{g} \in \cG}\|\tlt - \mathfrak{g} \tlp \|_2^2 \le  \|\tlt - G \tlp\|_2^2 \le 2\|\t - G \phi\|_2^2 + 2 \| \t^*-\phi^* \|_2^2.
\end{equation}
Taking minimum over $G \in \cG$ on the right hand side of \eqref{eq:summarize_dist_compare} and noting that $\t^*=\bt \ind$ (and similarly for $\phi^*$), we obtain 
\[\rho(\tlt,\tlp)^2 \le 2 \rho(\t,\phi)^2 + 2|\bt - \bp|^2 \|\ind\|_2^2,\]
as desired.
\end{proof}

\subsection{$L,s$ dependence of estimation rates}
It may be noted that, in the context of Proposition \ref{prop:CurveLB-normality} part (i),  if we have additional information on the dependence of $K_1(\sigma)$ on $L$ and/or $s$, then we can have informative asymptotic asymptotic upper bounds on $\varrho(\tlt_n,\t)$ vis-a-vis its dependence on $L$ and/or $s$.

\begin{proof}[Proof of Theorem \ref{cor:dependence}]
We begin by recalling that by the $\cG$-invariance of  $D_{KL}(p_\t \| p_\phi)$,  for $\varrho(\t,\t_0)$ (equivalently, $\|\t-\t_0\|_2$) small enough, we may assume without loss of generality that $\varrho(\t,\t_0)= \frac{1}{\sqrt{L}}\|\t-\t_0\|_2$. Since $\|\tlt_n-\t_0\|_2 \to 0$ as $n \to \infty$, This will be true for $\varrho(\tlt_n,\t_0)$ with high probability.  

\textit{The dilute regime.} \quad
In view of  Lemma \ref{lem:lbound-beltway} and Proposition \ref{prop:linear_LB}, we obtain a local curvature estimate on $D_{KL}(p_\t\|p_{\t_0})$ as 
\[D_{KL}(p_\t\|p_{\t_0}) \ge C \cdot \frac{s}{L \sigma^4}\cdot \|\t-\t_0\|_2^2.\]  
Thus, we are in the setting of Proposition \ref{prop:CurveLB-normality} part (i) with $K_1(\sigma)=C \cdot \frac{s}{L \sigma^4}$. In view of \eqref{eq:asymptotics-1} and \eqref{eq:asymptotics-2},   we conclude that the limiting distribution of  $\sqrt{n}\varrho(\tlt,\t_0)/\sigma^2$ is stochastically dominated by a $C_1 \sqrt{\chi^2(s)/s}$ random variable, for a constant $C_1>0$. We conclude by noting that the latter random variable is $O_p(1)$.

\textit{The regime of moderate sparsity.} \quad
In the case of generic sparse symmetric signals, it may be seen from \eqref{eq:curv_LB_moderate} and Remark \ref{rem:L-free_LB} that if the support typically $s$-sparse and is $s^\tau$ cosine-generic, then for $\varrho(\t,\t_0)$ small we have \[\|\Delta_2(\t,\t_0)\|_F \ge \frac{s^{\tau-4}}{\sqrt{L}}\|\t-\t_0\|_2.\] Hence, by Theorem \ref{thm:BRW2}, for such signals we have \[D_{KL}(p_\t\|p_{\t_0}) \ge C \cdot \frac{s^{2(\tau-4)}}{L \sigma^4}\cdot \|\t-\t_0\|_2^2.\] Once again, this places us in the context of  Proposition \ref{prop:CurveLB-normality} part (i), with \[K_1(\sigma)=C \cdot \frac{s^{2(\tau-4)}}{L \sigma^4}.\] 
Furthermore,  in view of \eqref{eq:asymptotics-1} and \eqref{eq:asymptotics-2}, we may conclude that the  limiting distribution of  $\sqrt{n}\varrho(\tlt,\t_0)/\sigma^2$ is stochastically dominated by a $C_2  s^{4-\tau}\sqrt{\chi^2(s)}$ random variable, for a constant $C_2>0$. The latter random variable is $O_p(s^{4.5 - \tau})$.

We finally observe that two significant examples of generic sparse symmetric signals -- namely,  the Bernoulli-Gaussian distribution and the $N^{\mathrm{symm}}_{[-s,s]}(0,\zeta^2I)$ have supports that are typically $s$-sparse and constant times $s$-cosine generic. We provide the details in the case of the Bernoulli-Gaussian; the case of the  $N^{\mathrm{symm}}_{[-s,s]}(0,\zeta^2I)$ is similar. We invoke Lemma \ref{lem:cosgen-symBer} to conclude that the symmetric Bernoulli-Gaussian distribution with sparsity $s$ and variance $\zeta^2$ is typically $s$-sparse (with sparsity constants $(1/2,2)$) and $s/32$ cosine generic.  The analogous Lemma to be applied for the $N^{\mathrm{symm}}_{[-s,s]}(0,\zeta^2I)$ distribution is Lemma \ref{lem:cosgen-symdet}.

Thus, for these two signal distributions, $\tau=1$ in these settings in the context of the discussion immediately above.

In view of this fact, and the discussion above, the Bernoulli-Gaussian signal ensemble and the $N^{\mathrm{symm}}_{[-s,s]}(0,\zeta^2I)$ entail estimation rates that, upon scaling by $\sigma^2/\sqrt{n}$, are $O_p(s^{3.5})$. 
\end{proof}

\section*{Acknowledgements} SG was supported in part by the MOE grants  R-146-000-250-133, R-146-000-312-114 and MOE-T2EP20121-0013. PR was supported by the NSF awards DMS-1712596,  IIS-1838071, DMS-2022448, and DMS-210637. The authors would like to thank Victor-Emmanuel Brunel for stimulating discussions that shaped the direction of this project, and Michel Goemans for pointing them to the partial digest problem. The authors are grateful to the anonymous referees for their meticulous reading of the manuscript and their prescient suggestions towards its improvement, and especially for pointing out important connections of the present work to the problem of crystallographic phase retrieval.



\newpage

\appendix 

\begin{center}
{\Large   \textbf{Appendix}}
\end{center} 
 
\section{Appendix : additional notations} \label{a:notn}
\begin{defi} \label{def:thetap}
Let $\{X_n\}_{n \ge 1}$ be a sequence of non-negative random variables, and $\{a_n\}_{n \ge 1}$ is a sequence of positive numbers (deterministic or random). Then:
\begin{itemize}
\item By the statement $X_n=O_p(a_n)$ we mean that, for every $\eps>0$, there exists  $0<C(\eps)<\infty$ such that \[ \liminf_{n \to \infty}\P \left[ X_n/a_n \le C(\eps) \right] \ge 1-\eps. \]
\item By the statement $X_n=\Omega_p(a_n)$ we mean that, for every $\eps>0$, there exists  $0<c(\eps)<\infty$ such that \[ \liminf_{n \to \infty}\P \left[ X_n/a_n \ge c(\eps) \right] \ge 1-\eps. \]
\item By the statement $X_n=\Theta_p(a_n)$ we mean that for every $\eps>0$, there exist $0<c(\eps)<C(\eps)<\infty$ such that \[\liminf_{n \to \infty}\P \left[ c(\eps) \le X_n/a_n \le C(\eps) \right] \ge 1-\eps.\]
\end{itemize}
\end{defi}
Further,  $\|\cdot\|_F$ will denote the Frobenius norm of a matrix, and the expectation $\E_G$ will be taken with respect to $G$ chosen uniformly from the group of isometries $\bG$. 

For any positive integer $m$, by the symbol $[m]$ we denote the set $\{1,\cdots,m\}$.

For two sequences of positive numbers $(a_k)_{k>0}$ and $(b_k)_{k>0}$, we write $a_k \ll b_k$ when we have $b_k/a_k \to \infty$ as $k \to \infty$. 

A sequence of events $\{E_m\}_{m \ge 1}$, defined with respect to probability measures $\P_m$, is said to occur with high probability if $\P_m[E_m] \to 1$ as $m \to \infty$.

For any $\t  = (\t_1,\ldots,\t_L) \in \R^L$, we denote $\bt=\frac{1}{L}\sum_{i=1}^L \t_i$.

\section{Appendix: Bernoulli-Gaussian distributions} \label{a:BG}
We define the notion of the Bernoulli-Gaussian distribution, and the symmetric version thereof. For that, we first define the notion of a Gaussian distribution indexed by a subset of $\Z_L$.

\begin{defi}[Subset-indexed Gaussian distributions] \label{def:Normal-subset}
Let $A \subset \Z_L$, $\mu:\Z_L \to \R$ a function supported on $A$ and $\Sigma$ be a positive definite $|A|\times |A|$ matrix. Then the Gaussian distribution indexed by $A$ with mean $\mu$ and covariance $\Sigma$, denoted $N_A(0,\Sigma)$, is the random vector $(\eta_k)_{k \in \Z_L}$, with $\eta_k=0$ for $k \in A^\c$, and $(\eta_k)_{k \in A}$ is the $|A|$-dimensional Gaussian random vector with mean $\mu$ and covariance $\Sigma$.
\end{defi}

This allows us to define the Bernoulli-Gaussian distribution, a key property of which is that the support is chosen at random according to a Bernoulli sampling scheme.
\begin{defi}[Bernoulli-Gaussian distribution] \label{def:BG}
Let $s \in [L]$  and $\Xi \subset \Z_L$ be a random subset obtained by selecting each member of $\Z_L$ independently with probability $s/L$.
The Bernoulli-Gaussian  distribution on $Z_L$ with variance $\zeta^2$ and sparsity $s$ is then defined as the Gaussian distribution indexed by $\Xi$ with mean $\mathbf{0}$ and covariance $\zeta^2 I$; in other words the random variable $N_\Xi(\mathbf{0},\zeta^2 I)$,  with the Gaussian entries being statistically independent of the support $\Xi$.
\end{defi}

Next, we introduce the concept of a standard symmetric Gaussian random variable indexed by a subset of $\Z_L$. To introduce the notion of a symmetric signal, we first recall the notion of the \textit{standard parametrization} of $\Z_L$ \eqref{eq:std-par}. 

We are now ready to define
\begin{defi}[Symmetric subset-indexed Gaussian distributions] \label{def:symm-Normal-subset}
Let $\Z_L$ be in the standard enumeration \eqref{eq:std-par},  let  $A \subset \Z_L$ be symmetric, i.e. $A = -A$ and let $\rho>0$. Let $A_+:=\{0,\ldots,\lfloor(L-1)/2 \rfloor\} \cap A$, and let $(X_k)_{k \in \Z_L}$ denote the random variable $N_{A_+}(0,\zeta^2 I)$. Then the  symmetric Gaussian distribution indexed by $A$ with mean $0$ and variance $\zeta^2$, denoted $N_A^{\symm}(0,\zeta^2 I)$, is the random vector $(\eta_k)_{k \in \Z_L}$ with $\eta_k=X_{|k|}$.
\end{defi}

Finally, all of the above taken together allows us to define 
\begin{defi}[Symmetric Bernoulli-Gaussian distribution] \label{def:symm-BG}
Let $\Z_L$ be in the standard enumeration \eqref{eq:std-par}. Let $\Xi_0 \subset \Z_L^+=\{0,\ldots,\lfloor(L-1)/2\rfloor \}$ be a random subset obtained by selecting each member of $\Z_L^+$ independently with probability $s/L$, and consider the symmetric subset $\Xi:=\Xi_0 \cup (-\Xi_0)$. Then the  symmetric Bernoulli-Gaussian distribution with mean zero, variance $\zeta^2$ and sparsity parameter $s$ is the distribution $N^{\symm}_\Xi(0,\zeta^2 I)$, with the Gaussian entries being statistically independent of the support $\Xi$.
\end{defi}
Heuristically, the symmetric Bernoulli-Gaussian distribution is obtained by taking a Bernoulli-Gaussian random variable on the positive part of $\Z_L$ and extending it to all of $\Z_L$ by making it symmetric about the origin.

\section{Appendix: Generic sparse signals} \label{a:gen_signal}

We  introduce the notions of signal support sets  that are \textit{typically $s$-sparse} and $\fb$-cosine generic.

\begin{defi} \label{def:typical-sparse}
Let $\a,\b>0$ be fixed numbers and $s \in [L]$ be a parameter that possibly depends on $L$. A probability distribution over subsets $\Xi \subset \Z_L$   is said to be \textit{typically $s$-sparse} with sparsity constants $(\a,\b)$ if we have  $\a \cdot s \le |\Xi| \le \b \cdot s$ with probability $1-o_L(1)$.
\end{defi}

To introduce the concept of cosine-genericity of a set, we first define the cosine functional of a set $\Xi \subset \Z_L$ for an element $a \in \Z_L$.

For $\Xi \subset \Z_L$ and $a \in \Z_L$, define
\begin{equation} \label{eq:def-cosine}
\cC(\Xi,a)=\mathbbm{1}_{\{0 \in \Xi\}} + 2 \sum_{k \in \Xi \setminus \{0\}} \cos^2(2\pi a k/L),
\end{equation}
where $\mathbbm{1}_A$
denotes the indicator function of the event $A$.

Then we are ready to introduce
\begin{defi} \label{def:cosgen}
Let $\fb>0$  be a parameter, possibly depending on $L$. A probability distribution over subsets $\Xi \subset \Z_L$   is said to be \textit{$\fb$-cosine generic} if, with probability $1-o_L(1)$, we have $\min_{a \in \Z_L} \cC(\Xi,a) \ge \fb (1-o_L(1))$. 
\end{defi}
Equivalently, we say that the random variable $\Xi$ is cosine generic with parameter $\fb$. Cosine genericity of a (random) set is a condition that aims to ensure that, with high probability, the set under consideration is sufficiently generic, in the sense that there are no specialised algebraic or arithmetic relations satisfied by the elements of the set which would make $\min_{a \in \Z_L} \cC(\Xi,a)$ small.

Putting all of the above together, we may introduce the generic $s$-sparse symmetric signals.
\begin{defi}\label{def:generic_symm_signal}
Let $s \in [L]$ be a parameter, possibly depending on $L$, and $\a,\b,\zeta,\tau>0$ be fixed. We call a random signal $\t:\Z_L \to \R$ to be a generic $s$-sparse symmetric signal with dispersion $\zeta^2$, sparsity constants $\a,\b$ and index $\tau$ if the following hold:
\begin{itemize}
\item The support $\Xi$ of $\t$ is typically $s$-sparse with sparsity constants $(\a,\b)$and $s^\tau$-cosine generic.
\item $\t \sim N_{\Xi}^{\symm}(0,\zeta^2 I)$, with the non-zero entries of $\t$ being statistically independent of $\Xi$.
\end{itemize}
\end{defi}

\section{Appendix: On the size of collision free sets} \label{a:coll_free} 

In this section, we provide detailed arguments for the assertions that the size of a collision-free subset  $A \subset \Z_L$ is maximally $O(L^{1/2})$ and typically $O(L^{1/3})$. 

To this end, we let $1 \le k \le L$, and we consider a subset $B \subset \Z_L$ of size $|B|=k$. If $B$ is collision-free, then $B$ entails $k(k-1)$ distinct differences between its points; we call this set of differences $D$. For $x \in \Z_L \setminus B$, we want to understand size restrictions on $|B|$ that enable $B \cup \{x\}$ to be a collision-free set. If $B \cup \{x\}$ has to be collision-free, we note that for any fixed $u \in B$, the difference $x-u$ needs to be $\notin D$. This rules out $k(k-1)$ choices for $x$. Thus, such a point $x$ can be found only if  $k(k-1) < L - k$, which gives us an upper bound of $k=O(L^{1/2})$, as desired.

We note in passing that the probability of a randomly selected $x$ in the above setting to yield a collision-free subset $B \cup \{x\}$ is bounded above by $(L-k-k(k-1))/L$, for \textit{any} set $B$.

Now we examine the largest value of $m$ for which a random subset drawn of size $m$ drawn from $\Z_L$ collision free with positive probability. For concreteness, we consider $m$ samples without replacement from $\Z_L$. 


For $1\le k \le m$, we denote by $\fS_k$ the set of first $k$ random samples without replacement. Then we may write
\begin{align*}
& \P[\fS_m ~\text{is collision-free}] \\
= & \P[\fS_m ~\text{is collision-free} ~| ~\fS_{m-1} ~\text{is collision-free}] \cdot \P[\fS_{m-1} ~\text{is collision-free}] \\
= & \prod_{k=1}^{m-1} \P[\fS_{k+1} ~\text{is collision-free} ~| ~\fS_{k} ~\text{is collision-free}] \\
= & \prod_{k=1}^{m-1} \P_{x \sim \text{Unif}(\Z_L \setminus \fS_{k})} \big[\fS_{k} \cup \{x\} ~\text{is collision-free} ~| ~\fS_{k} ~\text{is collision-free}\big] \\
\le & \prod_{k=1}^{m-1} \frac{L-k-k(k-1)}{L}  \quad  \text{[using the analysis for the set $B$ above]} \\
= & \prod_{k=1}^{m-1}  \left(1 - \frac{k^2}{L} \right) 
~\le   \prod_{k=1}^{m-1} \exp(-\frac{k^2}{L}) 
~\le  \exp(-c m^3/L).
\end{align*}

Thus, if $m^3/L \to \infty,  \P[\fS_k ~\text{is collision-free}] \to 0$. Therefore, for a random subset of size $m$ to be collision-free with positive probability, we must have $m=O(L^{1/3})$, and to have the same property with high probability, we must have $m=o(L^{1/3})$.


\bibliographystyle{aomalpha}

\bibliography{main}

\end{document}